\documentclass{amsart} 

\usepackage{etex}
\usepackage[utf8]{inputenc} 
\usepackage[english]{babel} 
\usepackage{anysize} 

\usepackage{amsthm}  
\usepackage{amsmath}
\usepackage{amsfonts}
\usepackage{amssymb}

\usepackage{multicol} 
\usepackage{graphicx} 
\usepackage[dvipsnames]{xcolor} 
\usepackage{pgf,tikz}
\usepackage{tikz-cd}
\usepackage{pictex}
\usepackage{mathrsfs}
\usepackage{enumerate}
\usepackage{verbatim} 
\usepackage{array}
\usepackage[pagebackref=true]{hyperref}

\hypersetup{
    colorlinks,
    linkcolor={blue!70!black},
    citecolor={blue!70!black},
    urlcolor={blue!70!black}
}

\usepackage{tikz}
\usepackage{centernot}
\usepackage{mathtools} 
\usepackage{enumitem} 

\usepackage{tikz-cd}									    					

\marginsize{2 cm}{2 cm}{1 cm}{2 cm}

\DeclareMathOperator{\im}{im}
\DeclareMathOperator{\coker}{coker}
\DeclareMathOperator{\RMod}{\mathbf{Mod}}
\DeclareMathOperator{\Mod}{G-\mathbf{Mod}}
\DeclareMathOperator{\EMod}{E-\mathbf{Mod}}
\DeclareMathOperator{\DMod}{\mathcal{D}-\mathbf{Mod}}
\DeclareMathOperator{\Hom}{Hom}
\DeclareMathOperator{\Ext}{Ext}
\newcommand{\Loc}{\Mod(U_p)}

\newcommand{\Locc}{{\Mod(U_p^*)}}
\newcommand{\As}{{C^*}}
\newcommand{\ddA}{C}

\DeclareMathOperator{\Hol}{{\Mod^{\mathrm{gfg}}}}

\newcommand{\ov}[1]{\overline{#1}}
\newcommand{\wt}[1]{\widetilde{#1}}
\DeclareMathOperator{\Homm}{\mathcal{H}om}

\DeclareMathOperator{\QCoh}{\mathbf{QCoh}}

\DeclareMathOperator{\Spec}{Spec}

\numberwithin{equation}{section}
\newtheorem{lemma}[equation]{Lemma}
\newtheorem{proposition}[equation]{Proposition}
\newtheorem{theorem}[equation]{Theorem}

\newtheorem{THEOREM}{Theorem}

\newtheorem{corollary}[equation]{Corollary}

\theoremstyle{definition}

\newtheorem{remark}[equation]{Remark}
\theoremstyle{definition}
\newtheorem{definition}[equation]{Definition}

\newtheorem{notation}[equation]{Notation}
\numberwithin{equation}{section}

\newcommand{\gap}[1]{{\color{OliveGreen} \sf $\clubsuit\overset{\bullet\ \cdot}{\frown}\clubsuit$ GAP: [#1]}}

\newcommand{\removed}[1]{{\color{orange} \sf $\clubsuit\overset{\bullet\ \cdot}{\frown}\clubsuit$ REMOVED: [#1]}}

\newcommand{\C}{\mathbb C}
\newcommand{\Z}{\mathbb Z}

\newcommand{\A}{\mathbb A}

\renewcommand{\O}{{\mathcal{O}}}

\renewcommand{\div}{\mathrm{div}}

\DeclareMathOperator{\Aut}{Aut}

\newcommand{\cA}{\mathcal{A}}
\newcommand{\cC}{\mathcal{C}}
\newcommand{\Id}{\mathrm{Id}}
\newcommand{\pz}{{U_p}}
\newcommand{\pzz}{{U_p^*}}

\DeclareMathOperator{\St}{St}
\newcommand{\citep}[1]{\cite{#1}}
\DeclareMathOperator{\supp}{supp}
\DeclareMathOperator{\Prin}{Prin}
\DeclareMathOperator{\Div}{Div}

\renewcommand{\P}{\mathbb P}

\title{The local information of equivariant sheaves and elliptic difference equations}
\author{Mois\'es Herrad\'on Cueto}
\address{Department of Mathematics, Louisiana State University, 303 Lockett Hall, Baton Rouge, LA 70803, USA.}
\email {moises@lsu.edu}\urladdr{http://www.math.lsu.edu/~moises}
\date{\today}


\begin{document}

\maketitle

\begin{abstract}
We study the singularities of algebraic difference equations on curves from the point of view of equivariant sheaves. We propose a definition for the formal local type of an equivariant sheaf at a point in the case of a reduced curve acted on by a group which is virtually the integers. We show that with this definition, equivariant sheaves can be glued from an ``open cover''. Precisely, we show that an equivariant sheaf can be uniquely recovered from the following data: the restriction to the complement of a point, the local type at the point itself, and an isomorphism between the two over the punctured neighborhood of said point.

We study symmetric elliptic difference equations (``elliptic equations'' from now on) from this point of view. We consider several natural notions for an algebraic version of symmetric elliptic difference equations, i.e. symmetric elliptic difference modules (``elliptic modules''). We show that different versions are not equivalent, but we detail how they are related: all the versions embed fully faithfully into the same category of equivariant sheaves. This implies that we can use the theory for equivariant sheaves to study singularities of elliptic equations as well.

One reason to study elliptic equations is that they generalize, and degenerate to, ($q$-)difference equations (i.e. equivariant sheaves) and differential equations (i.e. $D$-modules) on the projective line. We discuss this from the elliptic module point of view, which requires studying elliptic modules on singular curves. We study the relation between elliptic modules on singular curves and their normalization. We show that for modules which are flat at the singular points there is an equivalence and we give examples showing that this cannot be improved upon.
\end{abstract}

\tableofcontents

\section{Introduction}

This paper concerns equivariant sheaves on curves and their local study. Equivariant sheaves can be interpreted as an algebraic counterpart to discrete equations: these include difference equations, which are linear recurrence relations of the form $y(t+1) = A(t)y(t)$ for $A(t)\in \operatorname{GL}_n(\C(t))$ and $y$ is a column vector; and $q$-difference equations, which take the form $y(qt) = A(t)y(t)$ for a given $q\in \C^\times$. The relation between equivariant sheaves and discrete equations is analogous to the relation between $D$-modules and differential equations.

The goal of this paper is to provide a notion for the local data of an equivariant sheaf around the formal neighborhood of a point $p$ on a curve $C$ (this is Definition~\ref{def:mainDef}). We do so in the case where the group acting on the curve is an extension of a finite group by $\Z$. We show that this definition is reasonable in that a sheaf can be recovered from its restriction to the formal neighborhood around $p$, its restriction to the open set $C\setminus p$ and an isomorphism between these two modules on the intersection, i.e. the punctured formal neighborhood around $p$. This is the content of Theorem~\ref{thm:main}.

Let us state it precisely: We are given a group $G$ that has a finite index subgroup isomorphic to $\Z$, acting on a reduced curve $C$ over a field $k$, with a closed point $p$. We will focus on equivariant sheaves whose stalks at every generic point of the curve are finitely generated, and call the category they form $\Hol(C)$. The restriction $|_\pz$ we define lands in a category of modules on the formal neighborhood $\pz$ with extra structure, which we will call $\Mod(\pz)$. We may also consider sheaves on the ``open'' subscheme obtained by removing the orbit of $p$ from $C$, which we denote $\As $. Similarly we can restrict these to the punctured neighborhood $\pzz$ (see Definition~\ref{def:puncturedCurve}). The usual pullback of quasicoherent sheaves to an open set can be enhanced in a natural way for equivariant sheaves on (the formal neighborhood in) a curve. When we include these restriction functors we obtain a commutative (up to natural isomorphism) square of restrictions:

\begin{equation}\label{eq:fiberDiagramIntro}
\begin{tikzcd}[column sep = 3 em, ampersand replacement = \&]
\Hol(C)\arrow[r,"|_{C^*}"]\arrow[d,"|_\pz"] \&
\Hol(\As)\arrow[d,"|_\pzz"] \\
\Mod(\pz) \arrow[r,"|_{\pzz}"]\&
\Mod(\pzz).
\end{tikzcd}
\end{equation}

\begin{THEOREM}\label{thm:main}
The Diagram~(\ref{eq:fiberDiagramIntro}) is a cartesian square of categories. More explicitly, it induces an equivalence between $\Hol(C)$ and the category $\Loc\times_\Locc\Hol(\As)$. This is the category of triples $(M_\pz,M_\As,\cong)$, consisting of objects $M_\pz\in \Mod(\pz)$, $M_\As\in \Hol(\As)$ and a fixed isomorphism $M_\pz|_{\pzz}\cong M_\As|_\pzz$.
\end{THEOREM}

This theorem validates the definition of $|_\pz$ in that it ensures that at the very least no information is lost. It could be also interpreted as saying that $|_\pz$ provides a classification of singularities of discrete equations. 

An analogous theorem was proved previously in \cite{H}, in the case of difference equations on the affine line. The main difference in the current situation is that we are allowing group actions that are not necessarily free. This means that the definition of $|_\pz$ needs to be adapted to different situations. This theorem can also be thought of as analogous to the Beauville-Laszlo Theorem from \cite{beauvilleLaszlo}, in the equivariant situation.

All the relevant definitions and the proof of the Theorem can be found in Section~\ref{sec:mainThm}.

\subsection{Symmetric elliptic difference equations}

Symmetric elliptic difference equations (from now on, abbreviated to \textbf{elliptic equations}) are our main motivation to study discrete equations. They were introduced in \cite{Rains} in order to give an interpretation to the elliptic Painlev\'e equation arising in Sakai's classification of surfaces associated to difference Painlev\'e equations \cite{Sakai}. It was first shown that the differential Painlev\'e equations correspond to isomonodromy deformations of moduli spaces of differential equations \cite{Okamoto}, which are some of the surfaces in the classification. However, not all the surfaces in Sakai's classification arise in this way. Others arise as moduli spaces of discrete equations, such as difference equations \cite{AB}. Symmetric elliptic difference equations complete this picture by providing a moduli interpretation for the elliptic Painlev\'e equation. Concretely, one considers the moduli space of elliptic equations with certain prescribed singularities. This is a motivation to understand singularities of elliptic equations in general.

Elliptic equations arise as follows: discrete equations on the line take the form $y(\tau(x)) = A(x)y(x)$ for some automorphism $\tau$ of $\P^1$. For an elliptic equation, the role of $\tau$ is played by a correspondence in $\P^1\times \P^1$, i.e. a curve $E\subset \P^1\times \P^1$ which is required to have degree $2$ over each component and to be symmetric when the coordinates are interchanged. An elliptic equation is given by a matrix meromorphic function $A:E\to \operatorname{GL}_n(\C)$, and it takes the form $y(t) = A(s,t)y(s)$ whenever $(s,t)\in E$. The matrix $A$ is required to satisfy the relation $A(s,t)=A(t,s)^{-1}$. In this paper we elaborate on (symmetric) elliptic (difference) modules, the counterpart to $D$-modules for this setting.

In the case where $E$ is the union of the graphs of $\tau$ and $\tau^{-1}$ for $\tau\in \Aut(\P^1),\tau^2\neq \Id$, elliptic equations are $\tau$-difference equations on $\P^1$, which include usual difference equations and $q$-difference equations. Further, if $E$ is the nonreduced double diagonal, certain elliptic modules become equivalent to $D$-modules on $\P^1$ (Proposition~\ref{prop:elliptic-Dmod}). Part of the interest on elliptic equations is due to the fact that they can degenerate to all these situations, and that this explains degenerations between surfaces in Sakai's classification.

Elliptic equations can be interpreted as equations on $E$ rather than $\P^1$: the pullback $\wt y(s,t) = y(s):E\to \C^n$ of a solution satisfies the equations $\wt y (s,t)=\wt y(s,t')$ and $\wt y(t,s)=A(s,t)\wt y(s,t)$ for all $s,t,t'\in \P^1$. The involutions $(s,t)\mapsto (s,t')$ and $(s,t)\mapsto (t,s)$ generate a dihedral group $G$ acting on $E$, and the equations satisfied by $y$ can be thought of as describing the $G$-equivariance of $\wt y$. We will adopt the convention that $G$ is always the infinite dihedral group, even if its action on $E$ is not free.

For this reason, in Theorem~\ref{thm:main} we work with groups $G$ which contain $\Z$ as a finite index subgroup, since these include the infinite dihedral group. There are other groups acting on curves that we are not treating in the present paper. For example, if $C$ is a group ($C=\mathbb A^1, \mathbb A^1\setminus 0$ or an elliptic curve), one may consider $r$ many elements in the group, whose translations give rise to an action of $\Z^r$. We believe these generalizations to be akin to letting $C$ be a higher dimensional variety. One reason for this belief is that in characteristic 0 the Fourier transform of \cite{L} is an equivalence between such equivariant sheaves and sheaves on $(\A^1\setminus 0)^r$ with extra structure in the affine case, and if we start with an elliptic curve, we obtain complexes of sheaves on a $(\A^1\setminus 0)^r$-bundle over the (dual) elliptic curve.

As in the case of differential and difference equations, we would like to construct some algebraic objects which encode elliptic equations. Due to the fact that the dihedral group's action on $E$ is not free, we have more than one reasonable way to define ``(symmetric) elliptic (difference) modules'' (see Remark~\ref{rem:possibleEllEq}). We make a choice (Definition~\ref{def:ellEq}), and we show how different options are related. We also show how the singularities of $E$ might come into play.

\begin{THEOREM}\label{thm:ellipticComparison}
Let $E\subseteq \P^1\times \P^1$ be a degree $(2,2)$ symmetric curve with no horizontal or vertical components. Let $\sigma: E\to  E$ be the automorphism interchanging the factors, and let $\sigma_1$ be the nontrivial deck transformation of $\pi_1: E\to \P^1$. Let $G$ be the infinite dihedral group generated by $\sigma$ and $\sigma_1$. Let $Z$ be the singular set of $E$. Let $\wt E$ be the normalization of $E$, and let $G$ also act on $\wt E$ by lifting the action on $E$.

\begin{enumerate}
\item The category of elliptic modules embeds fully faithfully into the category of $G$-equivariant sheaves on $E$.
\item The full subcategory of elliptic modules which are flat at $\pi_1(Z)\subset\P^1$ embeds fully faithfully into the category of $G$-equivariant sheaves on $\wt E$.
\end{enumerate}
\end{THEOREM}

Theorem~\ref{thm:ellipticComparison} is proven as Propositions~\ref{prop:EllipticModulesAreDihedral} and~\ref{prop:FlatEllipticModulesAreDihedral-normalization}, along with describing the image of the corresponding embeddings. It allows us to compare all the possible definitions in Remark~\ref{rem:possibleEllEq}. We would also like to remark that the second half of the statement cannot be improved to the whole category of elliptic modules: there is a functor defined on the whole category, but it is faithful and not full (Remark~\ref{rem:flatnessIsNeeded-notFull}). This is not surprising, since we use the results on \cite{Ferrand} relating quasicoherent sheaves on a curve and its normalization, which also require flatness.

Using this comparison, we can rephrase Theorem~\ref{thm:main} in the situation of elliptic equations. This is the content of Theorem~\ref{thm:mainThm-elliptic}.

In order to understand elliptic modules as a generalization of ($q$-)difference modules on $\P^1$, one must use the normalization of the singular curve which is the union of the graphs of $\tau$ and $\tau^{-1}$, which motivates the second part of Theorem~\ref{thm:ellipticComparison}. For completeness, we show how all the remaining situations in which the curve $E$ is not integral relate to modules on $\P^1$.



\subsection{Structure of the paper}

Section~\ref{sec:definitionsAndNotation} contains the notation and definitions used throughout the paper.

The main definition of the local information of a discrete equation is presented in Section~\ref{sec:mainDefs}, along with all the precise statements without proofs. All the relevant proofs for this section are in Section~\ref{sec:proof}.

Section~\ref{sec:bigSecElliptic} concerns elliptic modules and all the related notions. The relation between elliptic modules and equivariant sheaves is made precise in Section~\ref{sec:elliptic-equiv}, as well as the relation between elliptic modules on a curve and its localization. Section~\ref{sec:elliptic-discrete} gives explicit descriptions of elliptic modules whenever $E$ is not integral. This includes the relation with equations on the line such as difference equations and differential equations. Except for Theorem~\ref{thm:mainThm-elliptic}, Section~\ref{sec:bigSecElliptic} is independent of Section~\ref{sec:mainThm}.

In Section~\ref{sec:examples} we use the local type and the comparison theorems of Section~\ref{sec:bigSecElliptic} to describe the elliptic modules of generic rank 1, along with their local type.

\subsection*{Acknowledgments}

I am very grateful to Dima Arinkin for suggesting the problem and for many useful discussions. I also wish to thank Eva Elduque for helpful conversations and comments. This work was partially supported by National Science Foundation grant DMS-1603277. This material is based upon work supported by the National Science Foundation under Grant No. 1440140, while the author was in residence at the Mathematical Sciences Research Institute in Berkeley, California, during the spring semester of 2019.

\section{Definitions and notation}\label{sec:definitionsAndNotation}

\subsection{Equivariant sheaves and discrete equations}\label{sec:introToequivariant}

Throughout, we work over a field $k$ of characteristic different from 2. All the sheaves we will consider are quasicoherent. Our main objects of study are equivariant sheaves. Let us recall the definition for convenience.

\begin{definition}\label{def:equivariant}
Let $G$ be a (possibly formal) group acting on a scheme $C$ by a map $\alpha:G\times C\to C$. A ($G$-) equivariant sheaf is a sheaf $M\in \QCoh(C)$ together with an isomorphism $\cA:\pi_C^*M \overset{\sim}\to \alpha^*M$, satisfying the following cocycle condition on $G\times G\times C$: $(m\times \Id_C)^*\cA = (\Id_G\times \alpha)^*\cA \circ \pi_{23}^*\cA$. Here $m$ is the multiplication on $G$ and $\pi_{23}$ is the projection onto $G\times C$ that forgets the first factor. Further, if we let $i\times \Id:C\to G \times C$ be the inclusion of the identity of $G$, we must have that $(i\times \Id)^*\cA = \Id_M$. Morphisms of equivariant sheaves $\Hom^G((M,\cA_M),(N,\cA_N))$ are morphisms of sheaves $\phi:M\to N$ such that $\alpha^*\phi\circ \cA_M = \cA_N\circ \pi_C^*\phi $.
\end{definition}

We will only consider discrete groups $G$ (as a formal scheme, $G \cong \bigsqcup_{g\in G} \Spec k$). In this case, $G\times C = \bigsqcup_{g\in G} C$, so an equivariant sheaf consists of the data of $M\in \QCoh (C)$, together with $\cA_g:M\overset{\sim}\to g^*M$ for every $g\in G$. The cocycle condition becomes the relation $\cA_{g_1g_2} = g_2^*\cA_{g_1}\circ \cA_{g_2}$, and the condition at the identity becomes $\cA_1=\Id$. A morphism of sheaves $\phi$ in this situation is a morphism of equivariant sheaves if and only if for every $g\in G$, $\cA_g\circ \phi = g^*\phi\circ \cA_g$.

Given an equivariant sheaf $M$, we can consider for $g\in G$ the map $(g^{-1})^*\circ \cA_g:M\to M$, which we will simply denote by $\ov g$. This is not a map of sheaves: rather, for every open set $U\subset C$, $\cA_g$ maps $M(U)$ to $g^*M(U)$, and $g^*$ identifies $g^*M(U)$ with $M(gU)$, so $\ov g$ maps sections on $U$ to sections on $gU$. It is also not $\O$-linear, like $\cA_g$ is, but rather if for a local function $f\in \O(U)$ we denote $f^g=f\circ g\in \O(g^{-1}U)$ (this is the right action of $G$ on $\O$), we have the relation
\[
\ov g (fs) = (f\circ g^{-1})\cdot \ov gs = f^{g^{-1}} \ov gs\in M(gU).
\]
We can interpret this as the relation $\ov g f = f^{g^{-1}} \ov g$, or $f\ov g = \ov gf^g$.
 Using this notation, the relation $\cA_{g_1g_2} = g_2^*\cA_{g_1}\circ \cA_{g_2}$ becomes $\ov {(g_1 g_2)}=\ov g_1\circ \ov g_2$: note that for a morphism of sheaves $\phi$, $g^*\phi = g^*\circ \phi \circ (g^{-1})^*$. Therefore, 
\[
\ov {g_1g_2} = (g_1^{-1})^*\circ (g_2^{-1})^*\circ \cA_{g_2g_1} =
 (g_1^{-1})^*\circ (g_2^{-1})^*\circ g_2^*\cA_{g_1}\circ \cA_{g_2}=
(g_1^{-1})^*\circ\cA_{g_1} \circ (g_2^{-1})^*\circ \cA_{g_2} = \ov g_1\circ \ov g_2.
\]
And the same reasoning shows that if $\ov {g_1g_2}=\ov g_1\circ \ov g_2$, then the maps $\cA_g = g^*\circ \ov g$ indeed define an equivariant structure on the sheaf $M$. Using this notation, a morphism of sheaves $\phi$ is a morphism of equivariant sheaves if and only if for every $g\in G$, $\ov g\circ \phi=\phi\circ \ov g$.

In particular, if $G$ is given by generators and relations, the equivariant structure is determined by a collection of isomorphisms $\{\cA_g:M\to g^*M\}$ for $g$ in a generating set of $G$, and a collection of isomorphisms $\{\cA_g\}$ for $g$ in a generating set will determine an equivariant structure if and only if for every relation $g_1\cdots g_m=1$, the corresponding map $\ov g_1\circ \cdots  \circ \ov g_m:M\to M$ is the identity (note that since $g_1\cdots g_m=1$, in this case the map will be an $\O$-linear isomorphism of sheaves).

If the group action is not faithful, we must take care to note which group the equivariant structure is for. For instance, given an automorphism $g$ of $C$ such that $g^2=\Id$, any isomorphism $\cA_g:M\to g^*M$ will give rise to a $\Z$-equivariant structure, where $\Z$ is generated by $g$. However, to obtain a $\Z/2\Z$-equivariant structure, we must also have the relation $\Id = \cA_{g^2} = g^*\cA_g\circ \cA_g$.

\subsubsection{Relation to discrete equations}

Linear recurrences give rise to equivariant sheaves: a linear recurrence for a group action takes the form $s(gx) = A_g(x) s(x)$ for all $g\in G$, where $s$ is a column vector and $A_g$ is an invertible matrix, of size $r$. We must have the conditions that $A_{g_1g_2}(x) = A_{g_1}(g_2x) A_{g_2}(x)
$, and $A_{1}(x)=\Id$. We may construct an equivariant sheaf by interpreting the recurrence as generators and relations: start with a free $\O$-module $F$ with basis $\{s_{i,g}\}$ parametrized by $1\le i\le r$ and $g\in G$. Let $F$ have the equivariant structure given by $\ov g s_{i,h}= s_{i,hg^{-1}}$. Let $A_g = (a_g^{ij})$. We consider the subsheaf $K$ of $F$ generated by the elements
$
\{s_{i,gh}-\sum_{j} a^{ij}_g s_{j,h}\}
$ for all $g,h\in G$ and $i$. Then $M=F/K$ is the desired equivariant sheaf (notice that the equivariant structure preserves $K$). In the category of equivariant sheaves it correpresents the functor of solutions to the recurrence. For example, if the scheme is a complex variety, maps from $M$ to the sheaf of meromorphic functions are the set of meromorphic solutions to the recurrence. Indeed, if $s_{i,1}$ map to certain functions $f_i(x)$, then $\ov{g^{-1}} s_{i,1}=s_{i,g}$ must map to $\ov{g^{-1}} f_i(x) = f_i(gx) $. Therefore the relation $s_{i,gh}=\sum_{j} a^{ij}_g s_{j,h}$ implies that $f_i(ghx)=\sum_{j} a^{ij}_g f_{j}(hx)$. Conversely, any solution to the recurrence will yield a morphism of sheaves using these formulas.

\subsection{Definition of elliptic equations}\label{sec:ellipticEqsDef}


Discrete equations on $\P^1$ are recurrences for an automorphism $\tau$. The jump to elliptic equations consists on replacing a map $\tau$ by a correspondence. Concretely, to mimic the correspondence $\tau \cup \tau^{-1}$, the correspondence must be of degree $2$ and symmetric. Precisely, $E$ must be a curve in $ \P^1\times \P^1$ which has degree $(2,2)$ and is symmetric, i.e. if $\sigma:\P^1\times \P^1$ is the map interchanging the factors, $\sigma(E) = E$. If $E$ is smooth and it has a $k$-valued point, it is an elliptic curve, hence the name. We define elliptic modules to capture the equations on $\P^1$ of the form $f(y) = A(x,y) f(x)$ for $(x,y)\in E$. For discrete equations, the relations induced by $\tau$ and $\tau^{-1}$ must be the same. Similarly, for elliptic equations we must have $A(x,y) = A(y,x)^{-1}$. As stated before, this will be the only kind of elliptic difference equations we refer to in this paper, so we will just refer to them as elliptic equations.

\begin{definition}\label{def:ellEq}
Let $E\subset \P^1\times \P^1$ be a degree $(2,2)$ symmetric curve. Let $\pi_1,\pi_2:E\to \P^1$ be the projections and let $\sigma:E\to E$ be the automorphism interchanging the factors. We assume that the projections $\pi_i$ are finite, i.e. $E$ has no vertical components.

An ($E$-)\textbf{elliptic module}, is a quasicoherent sheaf $M$ on $\P^1$, together with an isomorphism $
\mathcal A: \pi_1^* M \to \pi_2^* M
$, subject to the condition that $\sigma^*\mathcal A = \mathcal A^{-1}$.

We denote the category of $E$-elliptic modules as $\EMod$. A morphism $\phi\in \Hom_{\EMod}(M,N)$ of elliptic modules is a morphism $\phi$ of sheaves on $\P^1$ such that $\cA\circ \pi_1^*\phi = \pi_2^*\phi\circ \cA$.
\end{definition}

At the level of stalks, $\mathcal A$ is an isomorphism $\mathcal A_{x,y}:M_x\to M_y$ whenever $(x,y)\in E$, and $\mathcal A_{x,y} = \mathcal A_{y,x}^{-1}$. These should properly be called \textbf{symmetric elliptic difference modules}, for the following reason: Elliptic difference modules are sheaves on an elliptic curve $E$ equivariant under the translation by a specified point on $E$. In our situation, if we choose the origin of $E$ to be a ramified point of $\pi_1$, then $\pi_1$ identifies every point on $E$ with its opposite according to the group law of $E$. Since (symmetric) elliptic modules come from $\P^1$, their stalks at both points on the fibers of $\pi_1$ are identified, hence the name symmetric. A precise statement is provided by Proposition~\ref{prop:EllipticModulesAreDihedral}.

\begin{remark}
This is not the only possible notion of elliptic modules. If one adopts the idea that the unknown in an elliptic equation $f(x)=A(x,y)f(y)$ is a function $g(x,y)=f(x):E\to k$, then one should study sheaves on $E$. We could have defined an elliptic module as follows: let $\sigma:E\to E$ interchange the factors, let $\sigma_1$ be the automorphism $(x,y)\to (x,y')$ and let $G$ be the infinite dihedral group they generate. An alternative definition is as $G$-equivariant sheaves on $E$. We will see in Proposition~\ref{prop:EllipticModulesAreDihedral} that elliptic modules as we have defined them are a full subcategory of this category.
\end{remark}

\section{The local type}\label{sec:mainThm}
\subsection{Notation}

We consider the action of a group $G$ that is an extension of a finite group by $\Z$. Note that this includes all groups $G$ containing subgroups $H_1\triangleleft H_2 \triangleleft G$ such that $H_1$ and $G/H_2$ are finite and $H_2/H_1\cong \Z$: the projection $H_2\to\Z$ necessarily has a section, so $\Z$ is a finite index subgroup of $G$, and some finite index subgroup of this $\Z$ is a normal subgroup of $G$. Throughout, we will let $\tau\in G$ be a chosen generator of a normal finite index subgroup isomorphic to $\Z$.

Throughout we will let $C$ be a (possibly singular, possibly reducible, reduced, quasiprojective) curve over $k$ with an action of $G$. 
We will study $G$-equivariant quasicoherent sheaves on $C$. We will say a sheaf $M$ is generically finitely generated if the stalks at every generic point of $C$ are finitely generated, or equivalently if it contains a coherent sheaf $L$ such that $M/L$ is supported only on closed points. We denote the category of equivariant sheaves by $\Mod(C)$, and the full subcategory of generically finitely generated elliptic modules by $\Hol(C)$.

\subsection{Definitions}\label{sec:mainDefs}


We will let $p\in C$ be a closed point, and $\St_p<G$ be its stabilizer (the stabilizer of the closed point, i.e. of the corresponding ideal). Depending on whether $\St_p$ contains an infinite order element, $\St_p$ is either finite or it has finite index. We let $\St_p^* = \{h\in \St_p:h\tau h^{-1} =\tau \}$. Note that either $\St_p^*=\St_p$ or it is a subgroup of index $2$. Throughout, we distinguish three cases:
\begin{enumerate}[label=(\roman*)]
\item $\St_p$ is finite and $\St_p^*=\St_p$.
\item $\St_p$ is finite and $\St_p^*\neq \St_p$.
\item $[G:\St_p]$ is finite.
\end{enumerate}

Note that in situations (i) and (ii), $p$ must be a smooth point, as it has an infinite orbit.

\begin{definition}\label{def:puncturedCurve}
We let $\As = C\setminus Gp$. $\Mod(\As)$ is defined as the full subcategory of $\Mod(C)$ on which functions vanishing only on $Gp$ act as units, or equivalently as the category of $G$-equivariant quasicoherent sheaves on $C\setminus Gp$. The full subcategory of generically finitely generated modules is denoted $\Hol(\As)$. We denote the forgetful functor $j_*:\Mod(\As)\to \Mod(C)$, and we use the same notation for its restriction $\Hol(\As)\to \Hol(C)$.

The pushforward $j_*$ has a right adjoint $j^*$, which is given by pullback of quasicoherent sheaves to $\As$, endowed with the natural $G$-action.
\end{definition}

In what follows, we will let $R_p$ be the complete local ring at $p$, a local ring of dimension $1$, and $U_p=\Spec R_p$. We will let $K_p$ be its total ring of fractions, i.e. the direct sum of the function fields of its minimal primes. If $R_p$ is a domain, for example if $p$ is smooth, then $K_p$ is the fraction field of $R_p$. Note that $\St_p$ acts on $\Spec R_p$ by restricting the action on $C$, so we may talk of $\St_p$-equivariant modules on $U_p$.

\begin{definition}
The category of local types of equivariant sheaves is defined as follows, in cases (i), (ii) and (iii) above:
\begin{enumerate}[label=(\roman*)]
\item $\Mod(\pz)$ is the category of $R_p$-modules $M$, together with the additional information of two finite rank free submodules $M^l,M^r\subseteq M$, such that $M/M^l,M/M^r$ are supported on $p$. Additionally, $M$ is $\St_p$-equivariant, and the action of $\St_p$ preserves $M^{l}$ and $M^{r}$. Morphisms in $\Mod(\pz)$ are morphisms of equivariant $R_p$-modules which preserve the chosen submodules.
\item $\Mod(\pz)$ is the category of $R_p$-modules $M$, together with a single finite rank free submodule $M^{lr}$, as above, such that $M/M^{lr}$ is supported on $p$. Additionally, $M$ is $\St_p$-equivariant, and the action of $\St_p$ preserves $M^{lr}$. Morphisms are defined analogously.
\item $\Mod(\pz)$ is the category of $\St_p$-equivariant $R_p$-modules.
\end{enumerate}
\end{definition}

We will often write $M^\star$ to denote either one of $M^l$, $M^r$ or $M^{lr}$ to avoid repetition.

\begin{remark}
In cases (i) and (ii), $\Loc$ is not an abelian category, because not all morphisms have cokernels: a map $\phi:M\to N$ could have the property that $N^l/(\phi M^l)$ is not free, or the map into $N/\phi(M)$ might not be injective. However, it is an exact category, because it is a full subcategory of the abelian category of diagrams of equivariant $R_p$ modules with no restrictions about the arrows being injective or the modules being free, and it is closed under extensions. In particular, short exact sequences in $\Loc$ are short exact sequences of $R_p$-modules $M_1\to M_2\to M_3$ for which all the sequences of the form $M_1^\star\to M_2^\star\to M_3^\star$ are exact; and whenever kernels or cokernels exist, they can be computed in the larger category of diagrams.
\end{remark}

\begin{definition}
We define the categories of punctured local types of equivariant sheaves as the full subcategory $\Mod(\pzz)\subset \Mod(\pz)$ consisting of modules $M$ such that the natural map is an isomorphism $K_p\otimes_{R_p} M \cong M$ of $R_p$-modules. The forgetful functor will be denoted $j_*$.

The left adjoint to $j_*$ is denoted by $j^*$, and it is given by $j^*M = K_p\otimes_{R_p} M$, with $(j^*M)^\star$ defined to be the image of $M^\star$ inside of $j^*M$.
\end{definition}

We now define the restriction to the formal disk. From now on, we will denote $M_p = R_p\otimes M$ (where the tensor is over the stalk of $\O$ at $p$).

\begin{definition}\label{def:mainDef}
The restriction to the formal disk is defined in the following ways:
\begin{itemize}
\item[(i),(ii)] Let $M\in \Hol(C)$. Choose (arbitrarily) some coherent sheaf $L\subseteq M$ such that $M/L$ is supported on closed points. We define $M|_\pz\in \Mod(\pz)$ by $M|_\pz = M_p$. In case (i) we make $M|_\pz^l = (\ov\tau^n L)_p$ for $n\gg 0$, and $M|_\pz^r = (\ov\tau^{-n}L)_p$ for $n\gg 0$, and in case (ii) we let $M|_\pz^{lr} = (\ov\tau^n L)_p$ for $n\gg 0$. Then $\St_p$ acts on $M|_\pz$ via the restriction.
\item[(iii)] The restriction $|_\pz:\Mod(C)\to \Mod(\pz)$ consists of making $M|_\pz = M_p$ and restricting the action of $\St_p$.
\end{itemize}
\end{definition}

We show that the restriction is well-defined independently of choices in Proposition~\ref{prop:wellDefined}. Note that in cases (i) and (ii) $|_\pz$ is defined on generically finitely generated modules, while in case (iii) we can extend the definition to all modules. From now on, we will refer to $\Hol(C)$ in all three cases. All our proofs will extend to arbitrary modules in case (iii).

\begin{remark}
Note that $|_\pz$ maps $\Hol(\As)$ into $\Mod(\pzz)$. Further, the following square commutes (up to natural isomorphism).
\begin{equation}\label{eq:fiberDiagram}
\begin{tikzcd}[column sep = 3 em, ampersand replacement = \&]
\Hol(C)\arrow[r,"j^*"]\arrow[d,"|_\pz"] \&
\Hol(\As)\arrow[d,"|_\pz"] \\
\Mod(\pz) \arrow[r,"j^*"]\&
\Mod(\pzz)
\end{tikzcd}
\end{equation}
\end{remark}

\begin{theorem}\label{thm:fiberFinSt}[Theorem~\ref{thm:main}]
The diagram (\ref{eq:fiberDiagram}) is a cartesian square of categories.

More explicitly, it induces an equivalence between $\Hol(C)$ and the category $\Loc\times_\Locc\Hol(\As)$. This is the category of triples $(M_\pz,M_\As,\cong)$, consisting of objects $M_\pz\in \Mod(\pz)$, $M_\As\in \Hol(\As)$ and a fixed isomorphism $j^*M_\pz\cong M_\As|_\pz$. A morphism between two triples $f:(M_\pz,M_\As,\cong)\to (N_\pz,N_\As,\cong)$ is a pair of morphisms $f_\pz:M_\pz\to N_\pz$ and $f_\As:M_\As\to N_\As$ that commutes with the isomorphims.
\end{theorem}

Let us denote $\mathcal C = \Loc\times_\Locc\Hol(\As)$, and let $\Phi:\Hol(C)\to \mathcal C$ be the induced functor. In Section~\ref{sec:proof} we will build the necessary tools to construct an inverse to $\Phi$.

\begin{remark}
In case (iii), Theorem \ref{thm:fiberFinSt} holds after replacing $\Hol(C)$ by $\Mod(C)$. We do not use the generically finitely generated assumption anywhere, except to be able to define $|_\pz$ in cases (i) and (ii). We will keep referring to $\Hol(C)$ everywhere to simplify notation.
\end{remark}

\subsection{Proof of the main theorem}\label{sec:proof}
\begin{proposition}\label{prop:wellDefined}
The functor $|_\pz$ has the following properties. Note that in case (iii) all of them are clear or vacuous. We use $\star$ to denote any of $l$, $r$ or $lr$.
\begin{enumerate}
\item Its definition has no ambiguity, i.e. $|_\pz$ doesn't depend on the coherent sheaf $L\subseteq M$ as long as $M/L$ is supported on closed points and $n$ is chosen to be big enough (depending on the choice of $L$). Further, $\St_p$ preserves $M|_\pz^{\star}$.
\item $M|_\pz^\star$ is indeed a finite rank free module and $M|_\pz/M|_\pz^\star$ is supported on $p$.
\item The functor $|_\pz$ maps morphisms in $\Hol(\ddA)$ to morphisms in $\Loc$, i.e. for a morphism $f:M\to N$, $f(M|_\pz^\star) \subseteq N|_\pz^\star$. Further, $f|_\pz$ is $\St_p$-equivariant.
\item It is an exact functor, in the sense of exact categories: it maps short exact sequences to short exact sequences.
\item Let $f:M\to N$ be a morphism in $\Hol(\ddA)$. Then the restriction $f|_\pz^{\star} :M^{\star}\to N^{\star}$ is a morphism of free $R_p$-modules that has constant rank, i.e. its cokernel is a free module. Further, $N^\star/f|_\pz M^\star$ embeds into $N/f|_\pz M$, so $\coker f$ is an object of $\Loc$.
\item $G$ preserves $|_{\pz}$ in the following sense: for every $g\in G$, the induced map $\ov g:M_p\to M_{gp}$ sends $M|_{\pz}^\star$ to $M|_{U_{gp}}^\star$ if $g\tau = \tau g$ and it interchanges $M|_{\pz}^l$ with $M|_{\pz}^r$ if $g\tau = \tau^{-1} g$. In both cases it preserves $M|_{\pz}^{lr}$.
\end{enumerate}
\end{proposition}
\begin{proof}
\begin{enumerate}
\item Let $L_1,L_2$ be two coherent subsheaves of $M$ such that $M/L_1,M/L_2$ are supported on closed points. For $i=1,2$, $(L_1+L_2)/L_i$ is a coherent sheaf supported on closed points, and hence a finite length sheaf. This implies that the stalks of $L_1$ and $L_2$ can only differ at a finite set of points. Also notice that $\ov g$ identifies $ L_{g^{-1}p}$ and $(\ov gL)_p $ as $\ov g$ identifies $M_{g^{-1}p}$ with $M_p$. Applying this to $L_1=L$ and $L_2=\ov\tau L$, $(\ov\tau^{\pm n} L)_p = (\ov\tau^{\pm n+1} L)_p$ for $n\gg 0$, which shows that the definition doesn't depend on $n$ as long as it is big enough. Similarly, if two different coherent subsheaves are chosen then their stalks will be equal at $\ov\tau^{n}p$ as long as $|n|\gg 0$.

Let us show the $\St_p$-invariance. Start by assuming that $h\in \St_p^*$. Then $\ov h L$ and $L$ agree away from a finite set, so we may assume $L$ is $\ov h$-invariant: since we are in situations (i) or (ii), $\St_p$ is finite, so for some $m$ we have $h^m=1$ and we can replace $L$ by $L+\ov h L+\cdots +\ov h^{m}L$. Then we can see that for $n\gg 0$,
\[
\ov h M|_\pz^{\star} = \ov h (\ov\tau^{\pm n} L)_p = (\ov h\ov\tau^{\pm n} L)_p \overset{h\in \St_p^*}=(\ov\tau^{\pm n}\ov h L)_p = (\ov\tau^{\pm n} L)_{ p} =M|_\pz^{\star}.
\]
Lastly, if $h\in \St_p^*\setminus \St_p$, as before we may assume $L$ is $\ov h$-invariant. Then,
\[
\ov h M|_\pz^{lr} = \ov h (\ov\tau^{n} L)_p = (\ov\tau^{- n}\ov h L)_p = (\ov\tau^{- n} L)_{ p} =(\ov\tau^{- n} L)_{h p} =(\ov h^{-1}\ov\tau^{- n} L)_{p} = (\ov\tau^n L)_p =M|_\pz^{lr}.
\]

\item Since $L$ is a coherent sheaf, its torsion has finite support, so $M|_\pz^\star=(\ov\tau^{\pm n} L)_p=(L)_{\tau^{\mp n}p}$ is torsion-free for $|n|\gg 0$, and it is finitely generated since $L$ is. Further, since $M/\ov\tau^{\pm n} L$ is torsion, $M|_\pz/M|_\pz^\star = (M/\ov\tau^{\pm n}L)_p$ is torsion as well.
\item Let $L\subseteq M$ be a coherent sheaf such that $M/L$ is supported on closed points. Then $f(L)\subseteq N$ is coherent, so we may choose some coherent $L'\supseteq f(L)$ such that $N/L'$ is supported on closed points. Then for $n\gg 0$, $f|_\pz(M|_\pz^\star) \coloneqq (\ov\tau^{\pm n} f(L))_p\subseteq (\ov\tau^{\pm n} L')_p \coloneqq N|_\pz^\star$. The equivariance of the map is straightforward.

\item Given a short exact sequence $0\to M\to N\to P \to 0$ in $\Hol(C)$, take a coherent $L_N\subseteq N$ such that $N/L_N$ is supported on closed points. Then both $L_M = L\cap M\subseteq M$ and $L_P = L_N/L_M\subseteq P$ have the analogous property in $M$ and $P$ respectively. The short exact sequence $L_N\to L_M\to L_P$ yields the desired statement after applying $\ov\tau^{\pm n}$ and taking formal fibers.

\item Decompose $f$ as an epimorphism followed by a monomorphism, so that we have two short exact sequences: $0\to \ker f\to M\to f(M)\to 0$ and $0\to f(M)\to N\to \coker f\to 0$. Then the exactness of $|_\pz$ implies that the cokernel of $f|_\pz$ is $(\coker f)|_\pz$, which is an object of $\Loc$, so in particular $N^\star/f(M^\star) = (\coker f)|_\pz^\star$ is free and it embeds into $N/f(M)=(\coker f)|_\pz$.

\item Let us show this for $\star=l$, and the other cases will be analogous. Let $L\subseteq M$ be a coherent subsheaf such that $M/L$ is supported on closed points, and let $\epsilon$ be such that $g\tau = \tau^\epsilon g$. Then, if $n$ is big enough (depending on both $L$ and $g$), we can use the already proved independence of $L$ to show:
\[
\ov g(M|_{\pz}^l) \overset{n\gg 0}= \ov g((\ov\tau^n L)_p) = (\ov g\ov\tau^n L )_{gp} = (\ov\tau^{\epsilon n} \ov g L )_{gp} \overset{n\gg 0}=\left\{\begin{array}{rcl}
M|_{U_{gp}}^l & \text{if} & \epsilon = 1;\\
M|_{U_{gp}}^r & \text{if} & \epsilon = -1.
\end{array}
\right.
\]
\end{enumerate}
\end{proof}


\begin{proposition}
$\mathcal C$ is an abelian category.
\end{proposition}
\begin{proof}
Let $\mathcal D(\pz)\supset \Loc$ be the abelian category of diagrams of $\St_p$-equivariant $R_p$-modules $M^l\to M\gets M^r$ or $M^{lr}\to M$, in cases (i) and (ii) respectively. Let $\mathcal D(\pzz)$ be the full abelian subcategory of diagrams such that $K_p\otimes_{R_p} M\to M$ is an isomorphism. Then $\wt {\mathcal C} = \mathcal D(\pz)\times_{\mathcal D(\pzz)} \Hol(C)$ is a fibered product of abelian categories, so it is an abelian category which contains $\cC$ as a full subcategory. Therefore, $\mathcal C$ is abelian if it contains kernels and cokernels for all its morphisms.

Consider a morphism $f=(f_\pz,f_\As):(M_\pz,M_\As,\cong_M)\to (N_\pz,N_\As,\cong_N)$ in $\cC$. We will often omit the reference to the isomorphism $j^*M_\pz\cong_M M_\As|_\pz$, and simply understand that these modules are identified. Similarly we will say that $j^*f_\pz=f_\As|_\pz$ for simplicity. Note that kernels in $\Mod(\pz)$ always exist, and  $\Mod(\As)$ is an abelian category. However, a priori it is not clear that $\coker f_\pz$ exists: it would require that $f_\pz^\star$ has constant rank (for the relevant choices of $\star=l,r,lr$). Now, by definition $M_\pz^\star = (j^*M_\pz)^\star$ and $f_\pz^\star = (j^*f_\pz)^\star$. Since $j^*f|_\pz=f_\As|_\pz$, applying Proposition~\ref{prop:wellDefined} to $f_\As$ we see that indeed $f_\pz^\star$ has constant rank as desired.

The kernel (resp. the cokernel) of the morphism $f$ is $(\ker f_\pz,\ker f_\As)$ (resp. $(\coker f_\pz,\coker f_\As)$). These are indeed objects of $\cC$, i.e. they agree on $\Mod(\pzz)$ by the isomorphism induced from $M$ (resp. $N$):
\[
j^*((\operatorname{co})\ker f_\pz) = (\operatorname{co})\ker j^*f_\pz \cong_{M } (\operatorname{co})\ker f_\As|_\pz =((\operatorname{co})\ker f_\As)|_\pz.
\]
\end{proof}

\begin{remark}\label{rem:recollInCurve}
Take $M\in \Mod(C)$. Then the kernel and cokernel of the adjunction map $M\to j_*j^*M$ are supported on $Gp$, since they vanish after applying $j^*$. Thus, we may fit any module $M$ in an exact sequence as follows:
\[
0\to i_!i^!M\to M\to j_*j^*M\to i_!R^1i_!M\to 0.
\]
Here $i^!$ is the left adjoint to pushforward $i_!$ from $Gp$, and $R^1i^!$ is the first derived functor of $i^!$. However, we don't require these facts so we will not prove them here, and we can take the above sequence as the definition of $i!i^!$ and $i_!R^1i^!$. We will let $\ov M$ be the image of $M\to j_*j^*M$. It can be characterized as the largest quotient of $M$ with no sections supported on $p$.
\end{remark}

\begin{remark}\label{rem:recollinC}
The category $\mathcal C$ has the same structure as $\Mod(C)$ from Remark \ref{rem:recollInCurve} above. The role of $j_*j^*$ is played by the functor $(j_*j^*,\Id):\Loc\times_{\Locc} \Hol(\As)\to \Loc\times_{\Locc} \Hol(\As)$, and modules ``supported at $p$'' are pairs $(M_\pz,0)\in \mathcal C$. The long exact sequence for $M = (M_\pz,M_\As)$ takes the form
\[
0\to (i_!i^!M_\pz,0) \to (M_\pz,M_\As) \to (j_*j^*M_\pz,M_\As) \to (i_!R^1i_!M_\pz,0)\to 0
\]
Where again $i_!i^!$ and $i_!R^1i^!$ can be characterized as the kernel and cokernel of the map $\Id \to j_*j^*$. In this case, $\ov M$ is again defined as the image of $M$ in $j_*j^*M$, and it is the largest quotient of $M$ such that $\ov M_\pz$ has no sections supported at $p$. We will use the notation $j^*$ and $i^!$ for $\mathcal C$ as well from now on.

\end{remark}

We will now construct an inverse to $\Phi:\Mod(C)\to \cC$. First, let us construct a functor $\iota_*:\Mod(\pz)\to \Mod(C)$, which will we will prove to be the right adjoint to $|_\pz$. Let $M\in \Mod(\pz)$ and an open set $V\subseteq C$. Also, from now on fix $\Xi\subset G$ a (finite) set of representatives of $G/\langle \tau\rangle$. We will distinguish our three cases: in case (i), we let
\[
\iota_*M(V) = \left\{(m_g)_{g\in G}:m_g\in \left\{\substack{M\text{ if }gp\in V\\
K_p\otimes M\text{ if }gp\notin V} \right. ;\substack{
\displaystyle m_{\gamma\tau^i}\in M_\pz^l\text{ for }i\ll 0,\forall \gamma\in \Xi \\
\displaystyle m_{\gamma\tau^i}\in M_\pz^r\text{ for }i\gg 0,\forall \gamma\in \Xi
},m_{gh^{-1}} = \ov hm_g\forall g\in G,h\in \St_p\right\}
\]
In case (ii), we let
\[
\iota_*M(V) = \left\{(m_g)_{g\in G}:m_g\in \left\{\substack{M\text{ if }gp\in V\\
K_p\otimes M\text{ if }gp\notin V} \right. ; m_{\gamma\tau^i}\in M_\pz^{lr}\text{ for }|i|\gg 0,\forall \gamma \in \Xi;m_{gh^{-1}} = \ov hm_g\forall g\in G,h\in \St_p
\right\}\] 
In case (iii):
\[
\iota_*M(V) = \left\{(m_g)_{g\in G}:m_g\in \left\{\substack{M\text{ if }gp\in V\\
K_p\otimes M\text{ if }gp\notin V} \right. ;m_{gh^{-1}} = \ov hm_g\forall g\in G,h\in \St_p
\right\}\]

We give $\iota_*M(V)$ the structure of an $\O(V)$-module by letting $f\in \O(V)$ act by $
f(m_g)_g = (f^g m_g)_g$. 
One checks that this definition indeed makes $\iota_*M$ into a quasicoherent sheaf, where the restriction maps are induced by the map $M|_\pz\to K_p\otimes M|_\pz$ (notice in particular that if $f$ is regular at $gp$, then $f^g$ is regular at $p$). Further, the condition $\ov hm_g = m_{gh^{-1}}$ is preserved by multiplication by $f\in \O$:
\[
m_{gh^{-1}} = \ov hm_{g} \Rightarrow f^{gh^{-1}} m_{gh^{-1}} = f^{g h^{-1}} \ov h m_g= \ov h (f^g m_g)
\]
We endow $\iota_*M$ with the following $G$-equivariant structure: for $g_0\in G$, we make $
\ov g_0(m_g)_g := (m_{g_0^{-1}g})_g = (m_g)_{g_0g}
$. As before, we can easily check that the condition $\ov hm_g = m_{gh^{-1}}$ is preserved. One checks that $\ov g_0 (\iota_*M(V)) = \iota_* M(g_0V)$, and further let us verify that the $G$-action is compatible with the $\O$-action: for $f\in \O$ and $g_0\in G$,
\[
f\ov g_0(m_g)_g = f(m_{g_0^{-1}g})_g = (f^{g}m_{g_0^{-1}g})_g =(f^{g_0g}m_g)_{g_0g} =\ov g_0(f^{g_0g}m_g)_g = \ov g_0f^{g_0}(m_g)_g.
\]
\begin{lemma}
The functor $\iota_*$, defined as above, is the right adjoint to $|_\pz$.
\end{lemma}
Note that $|_\pz$ is only partially defined, since its domain is $\Hol(C)$ rather than $\Mod(C)$. However, the notion of an adjoint makes sense since $\Hol(C)$ is a full subcategory: we mean that for $M\in \Hol(C)$ and $N\in \Loc$, there is a natural isomorphism 
\[\Hom_{\Loc}(M|_\pz,N)\cong \Hom_{\Mod(C)}(M,\iota_*N).
\]
\begin{proof}
Let $M\in \Hol(C)$ and $N\in \Mod(\pz)$, and let $\phi:M\to \iota_*N$ be a map in $\Mod(C)$. A local section $m\in M$ is mapped to a sequence $(\phi_g(m))_g$. Consider $\phi_e$, which maps the stalk of $M$ at $p$ to $N$, and is $\O_C$-linear. We must check that $\phi_e$ is $\St_p$-equivariant provided that $\phi$ is $G$-equivariant. Indeed, if $\phi(m) = (\phi_g(m))$ and $h\in \St_p$,
\[
(\phi_g(\ov h m))_g = \phi(\ov h m) = \ov h \phi(m) =\ov h(\phi_g (m) )_g = (\phi_{h^{-1} g} (m))_g \Rightarrow \phi_e(\ov h m) = \phi_{h^{-1}}(m) = \ov h \phi_e(m)
\]

Finally, we must check that $\phi_e$ maps $M_p^{\star}$ into $N^{\star}$. Let us show this in the case where $\star=l$, and the other situations will be analogous. There exists some coherent sheaf $L\subseteq M$ such that $(\ov\tau^n L)_p = M_p^l$ for every $n\ge 0$. Then $\phi_e(L_p)=\phi_e ((\ov\tau^n L)_p) = \phi_{\tau^{-n}}(L_p)$, which for $n\gg 0$ is contained in $N^l$. This is the case because the stalk at $p$ of $L$ is finitely generated, so we only need to use that $\phi_{\tau^{-n}}(m)\in N^l$ for $n\gg 0$ and $m$ in a finite generating set of $L$.

In the other direction, let $\psi:M|_\pz\to N$ be a map in $\Mod(\pz)$, i.e. a $\St_p$-equivariant map such that $\psi M_\pz^{\star}\subseteq N^{\star}$. We define the map $\phi:M\to \iota_*N$ as follows: on a local section $m$,
\[
\phi(m) =  (\psi((\ov g^{-1}m)_p) )_g.
\]
If $m$ is regular at $gp$, then $\ov g^{-1}m$ is regular at $p$, i.e. $\psi((\ov g^{-1}m)_p)$ is contained in $N$ rather than $K_p\otimes N$, so the map is well-defined as a map of sheaves. Further, we check that it is $\O$-linear and $G$-equivariant: if $f\in \O$ and $g_0\in G$,
\[
f\phi(m) = (f^g\psi(\ov g^{-1}m)_p)_g = (\psi (\ov g^{-1} fm)_p)_g = \phi(fm);
\quad 
\ov g_0\phi(m) = (\psi(\ov g^{-1}\ov g_0m)_p)_g =\phi(\ov g_0m).
\]
We must check that the image of $\phi$ is contained in $\iota_*N$: the condition $\ov hm_g=m_{gh^{-1}}$ amounts to $
\ov h\psi((\ov g^{-1}m)_p)  = \psi((\ov h \ov g^{-1}m)_p) 
$, which follows from the $\St_p$-equivariance of $\psi$. Lastly, we must see that for $n\ll 0$ and $\gamma^{-1}\in \Xi$, $\psi((\ov \tau^{-n}\ov \gamma m)_p)\in N^l$, and similarly for $N^r$. This is indeed the case, since $m$ is contained in some coherent sheaf $L$, such that $M/L$ is supported on closed points. For $n\ll 0$ and $\gamma^{-1}\in \Xi$, $( \ov\tau^{-n}\ov \gamma L)_p = M|_\pz^l$ (recall that $\Xi$ is a finite set), and therefore $(\ov\tau^{-n} \ov \gamma m)_p\in M|_\pz^l$, which $\psi$ maps into $N^l$ by assumption.

It is straightforward to check that these two maps are inverse natural bijections between $\Hom(M,\iota_*N)$ and $\Hom(M|_\pz,N)$.

\end{proof}


We now define $\Psi$, which we will prove is the inverse of $\Phi$. Let $M= (M_\pz,M_\As)\in \mathcal C$. The adjunction $j^*\vdash j_*$ yields a natural map $f_1:M_\pz\to j_*j^*M_\pz\cong M_\As|_\pz$ (recall that this isomorphism is part of the data of $M$). The adjunction $|_\pz\vdash \iota_*$ yields a map $f_2:M_\As \to \iota_*M_\As|_\pz$. We define $\Psi M$ as the equalizer of $\iota_*f_1$ and $f_2$, i.e.
\[
\Psi (M_\pz,M_\As) = \ker (\iota_*M_\pz\oplus M_\As \longrightarrow \iota_* (M_\As|_\pz)) =\ker (\iota_*M_\pz\oplus M_\As \longrightarrow \iota_* (j^*M_\pz))
\]

\begin{lemma}\label{lem:InversesAreAdjoints}
$\Psi$ is right adjoint to $\Phi$.
\end{lemma}
\begin{proof}
This follows formally from previous discussion. Let $N = (N_\pz,N_\As)\in \mathcal C$, and let $M\in \Hol(C)$. Then,
\begin{align*}
\Hom(M,\Psi(N)) &\cong \ker (\Hom(M,\iota_*N_\pz)\oplus \Hom(M,N_\As)\to \Hom(M,\iota_* (N_\As|_\pz) ) \cong \\
&\cong \ker \left(\Hom(M|_\pz,N_\pz)\oplus \Hom(j^*M,N_\As) \to \Hom(M|_\pz,N_\As|_\pz)\right) \cong &   |_\pz\vdash \iota_*\\
&\cong \Hom_{\mathcal C}( (M|_\pz,j^*M),(N_\pz,N_\As)) =&\text{Def. of }\cC\\ &= \Hom_{\mathcal C}(\Phi(M),N).
\end{align*}
\end{proof}

\begin{lemma}\label{lem:exactness}
$\Phi$ is exact and $\Psi$ is left exact. Further, the following short exact sequence remains exact on the right after applying $\Psi$:
\[
0\to i_!i^!M\to M\to \ov M\to 0.
\]

\end{lemma}
\begin{proof}
Since $\Phi$ is a left adjoint, it is right exact, so we only need to show that it preserves injections. This follows from the fact that both $|_\pz$ and $j^*$ are exact. Likewise, $\Psi$ is left exact due to being a right adjoint.

Let $M= (M_\pz,M_\As)\in \mathcal C$. We have the short exact sequence $i_!i^!M\to M\to \ov M$ from Remark~\ref{rem:recollinC}. Let us check that after applying $\Psi$ it remains exact on the right. A local section $m\in \Psi\ov M$ is a pair consisting of $(\ov m_g)_g\in \iota_* \ov M_\pz$ and $m_\As\in M_\As$, agreeing on $\iota_*j^*M_\pz$. A preimage of $m$ must be a pair $((\wt m_g)_g,m_\As)\in \Psi M\subseteq \iota_*M_\pz\oplus M_\As$, where $\wt m_g$ map to $\ov m_g$.

Let us construct such a preimage in case (i). Note that the induced map $M_\pz^{\star}\to \ov M_\pz^{\star}$ is an isomorphism, since it is the quotient of a finite rank free module by its torsion. Therefore, for $n\ll 0$ and $\gamma\in \Xi$, $\ov m_{\gamma\tau^n},\in \ov M^l_\pz$ has a unique preimage in $M^l_\pz$, and similarly for $M^r_\pz$ (taking $n\gg 0$). Let $\Theta^l\subset G$ be the subset of $g$'s such that $\ov m_g\in \ov M^l_\pz$, and similarly for $\Theta^r$. Since $\St_p$ preserves $M^\star_\pz$, it follows that $\Theta^\star\St_p = \Theta^\star$, and $G\setminus (\Theta^l\cup \Theta^r)$ is finite.

We choose $\wt m_g$ for all $g\in \Theta^l$ as the only preimage of $\ov m_g$ contained in $M_\pz^l$, and we make the analogous choice for $g\in \Theta^r\setminus \Theta^l$. By the uniqueness of the choice and the fact that $\St_p$ preserves $M_\pz^\star$ and $\Theta^\star$, it must follow that for $g\in \Theta^l\cup \Theta^r$ and any $h\in \St_p$, $\wt m_{gh^{-1}} = \ov h\wt m_g$. For $g\notin \Theta^r\cup\Theta^l$ (a finite set), we choose a set of representatives of $(G\setminus \Theta^r\cup \Theta^l)/\St_p$, and for these representatives $g$ we let $\wt m_g$ be an arbitrary preimage of $\ov m_g$ in $M_\pz$. The remaining $g$'s are chosen in the unique way that ensures the condition that $\wt m_{gh^{-1}} = h\wt m_g$.

This provides an element $(\wt m_g)\in \iota_*M_\pz$ mapping to $(\ov m_g)\in \iota_*\ov M_\pz$. We must check that the element $((\wt m_g),m_\As)$ is in $\Psi M$, i.e. that this pair agrees on $\iota_*j^*M|_\pz$. This is the case because the map $M_\pz\to j^*M_\pz$ factors through $M_\pz\to \ov M_\pz$, and it is given that $(\ov m_g)$ and $m_\As$ agree. We have thus produced a preimage of $m$ as we desired.

In case (ii), we proceed as in case (i), replacing $M^{lr}_\pz$ by $M^l_\pz$, and noting that defining $\Theta^{lr}$ analogously ensures that $G\setminus \Theta^{lr}$ is finite.

For case (iii), we choose a (necessarily finite) set of representatives of $G/\St_p$, and for these we arbitrarily choose a preimage $\wt m_g$ of $\ov m_g$. For the remaining $g$'s, we ensure that $\wt m_{gh^{-1}} = h\wt m_g$, which implies that $\wt m_{gh^{-1}}$ maps to $\ov m_{gh^{-1}}$. Then as before it will follow that $((\wt m_g),m_\As)\in \Psi M$, because the map to $j^*M_\pz$ factors through $\ov M_\pz$.


\end{proof}

We can finally show that $\Phi$ and $\Psi$ are mutual inverses.

\begin{proof}[Proof of Theorem \ref{thm:fiberFinSt}]

The adjunction yields natural transformations $\eta:\Phi \Psi\to \Id$ and $\epsilon:\Id\to \Psi\Phi$. Let us start by proving that $\epsilon$ is an isomorphism: let $M\in \Hol(C)$. The identity of $\Phi M = (M|_\pz,M|_\As)$ induces by the adjunction the map $M\to \Psi\Phi M$, which chasing the proofs above is given by
\[
\epsilon: M(U)\ni m\longmapsto \left(
((\ov g^{-1}m)_p)_g,m|_\As
\right)\in \Psi\Phi M \subset\iota_*(M|_\pz)\oplus M|_\As\subset \prod_{g} M|_\pz \oplus M|_\As.
\]
We will use the following exact sequences:
\[
0\to i_!i^!M\to M\to \ov M\to 0;\quad 0\to \ov M\to j_*j^*M\to i_!R^1i^!M\to 0.
\]
Applying $\Psi\Phi$, which is left-exact, we obtain the following diagrams with exact rows:
\[
\begin{tikzcd}[column sep = 2 em, ampersand replacement = \&]
0\arrow[r]\arrow[d]\& \ov M\arrow[r]\arrow[d,"\epsilon_{\ov M}"]\& j_*j^*M \arrow[r]\arrow[d,"(1)"]\& i_!R^1i_!M\arrow[d,"(2)"]\&
 0\arrow[r]\arrow[d]\& i_!i^!M\arrow[r]\arrow[d,"(3)"]\& M \arrow[r]\arrow[d,"\epsilon_M"]\& \ov M\arrow[r]\arrow[d,"\epsilon_{\ov M}"]\& 0 \&\\
0\arrow[r]\& \Psi\Phi \ov M\arrow[r]\& \Psi \Phi j_*j^*M\arrow[r]\& \Psi \Phi i_!R^1i_!M \&
 0\arrow[r]\& \Psi\Phi i_!i^! M\arrow[r]\& \Psi \Phi M\arrow[r]\& \Psi \Phi\ov M.\&  \\
\end{tikzcd}
\]
If arrows $(1)$ and $(2)$ are isomorphisms it will follow that $\epsilon_{\ov M}$ is an isomorphism as well. Further, if arrow $(3)$ is an isomorphism, the five-lemma implies that $\epsilon_M$ is an isomorphism as well. Putting everything together, to show that $\epsilon$ is an isomorphism it suffices to prove that $\epsilon$ is an isomorphism when restricted to the images of $i_!$ and $j_*$, i.e. to sheaves supported on $Gp$ and sheaves in $\Mod(\As)$.


Suppose $M\cong \iota_!\iota^! M$. Then, $M|_\As=0$, so $M|_\pz$ is supported on $p$, and we want to prove that $M\cong \iota_*(M|_\pz)$. It's a matter of writing out the definitions and using the fact that in cases (i) and (ii), $\iota_*M|_\pz$ is contained in the direct sum $\bigoplus_{g} M|_\pz \subset \prod_g M|_\pz$, as $M|_\pz^\star=0$.

If $M\cong j_*j^*M$, then $\epsilon$ is injective, since $m|_\As=m$. Now, consider an element $ n =((m_g)_g,m_\As)\in \Psi\Phi M$. We have that $n=\epsilon(m_\As)$, so $\epsilon$ is surjective.

It remains to prove that $\eta:\Phi\Psi\to \Id$ is an isomorphism. Starting with $M= (M_\pz,M_\As)\in \mathcal C$, $\eta$ is given by $\eta_\pz$ and $\eta_\As$ as follows:
\[
\eta_\pz:(\Psi M)|_\pz\ni ( (m_g)_g,m_\As)_p \mapsto m_e \in M_\pz	
\]
\[
\eta_\As:j^*(\Psi M)\ni j^*( (m_g)_g,m_\As) \mapsto m_\As \in M_\As
\]
We must check that they are both isomorphisms (as Lemma \ref{lem:InversesAreAdjoints} guarantees that they are well-defined and that they agree on $j^*(\Psi M)|_\pz$). We try the same strategy, with the analogous exact sequences as before (from Remark \ref{rem:recollinC}). Applying $\Phi \Psi$ we obtain the following diagrams.

\[
\begin{tikzcd}[column sep = 2 em, ampersand replacement = \&]
0\arrow[r]\arrow[from = d]\& \ov M\arrow[r]\arrow[from = d,"\eta_{\ov M}"]\& j_*j^*M \arrow[r]\arrow[from = d,"(1)"]\& i_!R^1i_!M\arrow[from = d,"(2)"]  \&
0\arrow[r]\& i_!i^!M\arrow[r]\arrow[from = d,"(3)"]
\& M \arrow[r]\arrow[from = d,"\eta_M"]
\& \ov M\arrow[r]\arrow[from = d,"\eta_{\ov M}"]
\& 0\\
0\arrow[r]\& \Phi\Psi \ov M\arrow[r]\& \Phi\Psi j_*j^*M\arrow[r]\& \Phi\Psi i_!R^1i_!M\&
0\arrow[r]\& \Phi\Psi i_!i^! M\arrow[r]\& \Phi\Psi M\arrow[r]\& \Phi\Psi\ov M\arrow[r]\& 0. \\
\end{tikzcd}
\]
The rows of these diagrams are exact. This time, for the second diagram we need Lemma \ref{lem:exactness}.

As before, if arrows $(1)$ and $(2)$ are isomorphisms it will follow that $\eta{\ov M}$ is an isomorphism as well, and if $(3)$ is an isomorphism as well, the five-lemma will imply that $\eta_M$ is an isomorphism as desired. So we only need to check that $\eta_M$ is an isomorphism for modules $M$ with $M_\As=0$ and for modules with $M_\pz\cong j_*j^*M_\pz$.


In the case of a module $M$ with $M_\As=0$, as before it suffices to write the map $\eta$: $\Psi M\cong \iota_*M_\pz$, so $\eta_\As=0$ and $\eta_\pz$ is the map $(\iota_*M_\pz)|_\pz\to M_\pz$, which can be directly verified to be an isomorphism. For a module of the form $j_*M$, we have that $\Psi j_* M \cong M_\As$, from the definition of $\Psi$ (since $(j_*M)_\pz\cong (j^*j_*M)_\pz$), and therefore indeed $\eta$ is an isomorphism.

\end{proof}

\section{Symmetric elliptic difference modules}\label{sec:bigSecElliptic}
\subsection{Elliptic modules as equivariant sheaves}\label{sec:elliptic-equiv}

Elliptic modules are very closely related to equivariant sheaves. Refer to Section~\ref{sec:ellipticEqsDef} for the notation and the definition. We will let $\pi_1,\pi_2:E\to \P^1$ be the projections and let $\sigma:E\to E$ be such that $\pi_1\circ \sigma=\pi_2$. Further, let $\sigma_i$ be the deck involution of the double cover $\pi_i:E\to \P^1$ (notice that $\sigma_2 = \sigma\sigma_1\sigma^{-1}$), which we show exists as part of Lemma~\ref{lem:z2-descent}. Elliptic modules come with a $\Z/2\Z$-equivariant structure (for the action of $\sigma$), and the fact that they are pulled back from $\P^1$ means they have another $\Z/2\Z$-equivariant structure, for the action of $\sigma_1$ (Lemma~\ref{lem:z2-descent}). Together, they form an equivariant structure for the infinite dihedral group $G=\langle \sigma,\sigma_1:\sigma^2=\sigma_1^2=1\rangle$.

Elliptic modules as we have defined them are not equivalent to $G$-equivariant sheaves on $E$, but they do embed into these. The reason for the difference lies in the fixed points of $\sigma_1$: sheaves equivariant under the $\Z/2$-action of $\sigma_1$ are sheaves on the stack quotient of $E$ by $\Z/2$, but $\P^1$ is just the coarse moduli space for this stack, and they differ exactly at the branch locus of $\pi_1$. The relation between these two is simple: sheaves that descend to $\P^1$ are the ones for which $\sigma_1$ acts as the identity on the (derived) fibers at ramified points. This is the content of Lemma~\ref{lem:z2-descent}. We now present the main three results of this section, followed by their proofs.

 
For the results of this section it is essential that the characteristic of $k$ is not 2, as well as for Theorem~\ref{thm:mainThm-elliptic}, since it depends on these statements.

\begin{lemma}\label{lem:z2-descent}
Let $C$ be a smooth connected curve, and let $\pi:C'\to C$ be a finite flat map of degree 2. In this situation, there is a deck involution $\sigma:C'\to C'$ such that $\pi\circ \sigma=\pi$. Let $i:Y\hookrightarrow C'$ be the fixed scheme of $\sigma$, i.e. $Y$ is cut out by the ideal sheaf $I_Y =\langle g-g^\sigma:g\in \O_{C'}\rangle$.

Then $\pi^*$ induces an equivalence between quasicoherent sheaves on $C$ and $\Z/2\Z$-equivariant sheaves $M$ on $C'$ such that $Li^*_Y\cA_\sigma=\Id$. Here $\cA_\sigma:M\to \sigma^*M$ denotes the equivariant structure, and by $Li^*_Y\cA_\sigma=\Id$ we mean that it agrees with the isomorphism $Li^*_Y\cong Li^*_Y\circ \sigma^*$ induced from $i_Y=\sigma\circ i_Y$.
\end{lemma}

\begin{proposition}\label{prop:EllipticModulesAreDihedral}
Let $E\subseteq \P^1\times \P^1$ be a degree $(2,2)$ symmetric curve with no horizontal or vertical components. Let $\sigma: E\to  E$ be the automorphism interchanging the factors, and let $\sigma_1$ be the deck transformation of $\pi_1: E\to \P^1$. Let $G$ be the infinite dihedral group generated by $\sigma$ and $\sigma_1$. Let $i:Y\hookrightarrow E$ be the subscheme fixed by $\sigma_1$, i.e. the scheme cut out by the ideal sheaf $I_Y = \langle f-f^{\sigma_1}:f\in \O_E\rangle$. Then there is an equivalence between the following categories:
\begin{itemize}
\item $E$-elliptic modules.
\item The full subcategory of $G$-equivariant sheaves on $E$ such that 
$Li^*_Y\cA_{\sigma_1} = \Id$, where $Li^*_Y$ denotes the derived restriction to $Y$.
\end{itemize}
The equivalence of categories maps an elliptic module $M$ to $\pi_1^*M$ with the equivariant structure such that $\cA_{\sigma}=\cA$ coming from the elliptic module structure, and $\cA_{\sigma_1}$ is provided by Lemma~\ref{lem:z2-descent}.
\end{proposition}

\begin{proposition}\label{prop:FlatEllipticModulesAreDihedral-normalization}
Let $E\subseteq \P^1\times \P^1$ be a \textbf{reduced} degree $(2,2)$ symmetric curve with no horizontal or vertical components. Let the field $k$ be perfect. Let $\pi:\wt E\to E$ be the normalization of $E$, let $\sigma:\wt E\to \wt E$ be the automorphism interchanging the factors, and let $\sigma_i$ be the deck transformation of $\pi_i\pi:\wt E\to \P^1$ (note that $\sigma_2=\sigma\sigma_1\sigma$). Let $G$ be the infinite dihedral group generated by $\sigma$ and $\sigma_1$. Finally, let $Z$ be the singular set of $E$ 
and let $i_Y:Y\hookrightarrow E$, resp. $i_{\wt Y}:\wt Y\hookrightarrow \wt E$  be the fixed scheme of $\sigma_1$.

The pullback $\pi^*$ induces an equivalence between the following categories:

\begin{itemize}
\item $E$-elliptic modules which are flat at $\pi_1(Z)\subset\P^1$.
\item The full subcategory of $G$-equivariant sheaves on $\wt E$ satisfying two conditions:
\begin{enumerate}
\item At the points of $\pi^{-1}(Z)$, the sheaves are flat.
\item $Li^*_{\wt Y}\cA_{\sigma_1} = \Id$.
\end{enumerate}
\end{itemize}
The equivalence of categories maps an elliptic module $M$ to $\pi_1^*M$ with the equivariant structure such that $\cA_{\sigma}=\cA$ coming from the elliptic module structure, and $\cA_{\sigma_1}$ is provided by Lemma~\ref{lem:z2-descent}.
\end{proposition}

\begin{proof}[Proof of Lemma~\ref{lem:z2-descent}]

Note that in the unramified case this boils down to \'etale descent for quasicoherent sheaves, \cite[\href{http://stacks.math.columbia.edu/tag/023T}{Tag 023T}]{stacks-project}.

Let us start by explicitly showing the existence of $\sigma$. Since $\pi$ is finite flat of degree $2$, $\pi_*\O_{C'}$ is a locally free $\O_C$-module of rank $2$. We will omit $\pi_*$ from the notation and just denote $\O_{C'}=\pi_*\O_{C'}$, since we will only talk of sheaves on $C$. Let us start by showing that $\O_{C'}/\O_C$ is locally free (the flatness implies that $\O_C\subset \O_{C'}$). Since $\O_C$ is (locally) a Dedekind domain, it suffices to show that it is torsion-free. Suppose it had torsion: let $y\in \O_{C'},a,b\in \O_{C}$ be such that $ay=b$. The ideal $(a,b)\subset \O_C$ is locally free, so passing to a smaller open cover, we can assume that it is principal: thus we may assume that $a=ca'$, $b=cb'$ and $(a',b')=\O_C$. The flatness of $\O_{C'}$ implies that $c\in \O_C$ is not a zero divisor, so we have that $a'y=b'$. Since $\O_{C'}$ is finite over $\O_C$, $y$ is integral over $\O_C$, i.e. there is a monic polynomial annihilating it: $\sum_{i=0}^n a_iy^i=0$, where $a_n=1$. Multiplying by $a'^n$, we have $\sum_{i=0}^n a'^{n-i}a_ib'^i=0$, which implies that $a'$ divides $b'^n$. The conditions that $(a',b')=\O_C$ together with $a'|b'^n$ imply that $a'$ is a unit in $\O_C$. Therefore, $y=b'a'^{-1}\in \O_C$. This shows that $\O_{C'}/\O_C$ is locally free.

We have that both $\O_{C'}$ and $\O_{C'}/\O_C$ are locally free (of ranks $2$ and $1$, respectively). Consider an open cover over which they are both free, and for each open set choose a lifting $y'\in \O_{C'}$ of a generator of $\O_{C'}/\O_C$. Then (on a fixed open set), $\{1,y'\}$ is a basis of $\O_C$. Therefore, $y'^2=ay'+b$ for some $a,b\in \O_C$. We replace $y'$ by $y=y'-a/2$, so that $y^2\eqqcolon x\in \O_C$. Thus we have shown that $\O_{C'}$ is locally of the form $\O_C[y]/(y^2-x)$, and as an $\O_C$-module it is $\O_C\oplus y\O_C$. The action of $\sigma^*$ is $\O_C$-linear and generated by $y\mapsto -y$. This action is independent of the choice of $y$: one checks directly that any other $\wt y\in\O_{C'}$ whose square is in $\O_C$ is an element of $\O_C\cdot y$, and therefore the $\sigma^*$-action is unique. Since this canonical action is preserved by localization, it can be glued over the open cover to yield the desired deck transformation. Notice that $\O_C=\O_{C'}^\sigma \coloneqq \{\alpha\in \O_{C'}:\alpha^\sigma=\alpha\}$.

Now that we know the global existence of $\sigma$, we can see that the equivariant pullback $\pi^*$ is a local construction on $C$. Therefore, it is enough to prove the statement on an open cover. From now on, we will assume $C=\Spec R$ is affine, and $S:=\O_{C'} = R[y]/(y^2-x)$ for some $x\in R$.


For an $R$-module $M$, $\pi^*M = M\oplus yM$, and the natural isomorphism $\sigma^*\pi^*\cong (\pi\sigma)^* = \pi^*$ is the equivariant structure given by $\cA_{\sigma}(m_1+ym_2) = \sigma^*(m_1-ym_2)=\sigma^*m_1+y\sigma^*m_2$, for $m_1,m_2\in M$. Therefore, on $\pi^*M/y\pi^*M$ we see that $\cA_\sigma$ induces the map $m\mapsto \sigma^* m$, while on $y^{-1}(0)\subseteq \pi^*M$, it induces the map $m\mapsto -\sigma^*m$, since $y^{-1}(0)\subseteq yM\subset \pi^*M$. Conversely, suppose we start with an $S$-module $N$ with an equivariant structure $\cA_\sigma$ such that $\cA_\sigma m =\sigma^*m$ on $N/yN$, and $\cA_\sigma m =-\sigma^*m$ for $m\in N$ such that $ym=0$. In this case, we may split $N$ into eigenspaces for $\ov \sigma=\sigma^*\circ \cA_\sigma$: the $\Z/2\Z$-equivariance exactly imposes the condition that $\ov \sigma^2=1$, hence the eigenvalues are contained in $\{\pm 1\}$. Let $N=N_+\oplus N_{-}$, where $N_\pm$ is the sub-$R$-module on which $\ov \sigma$ acts as $\pm \Id$. The above assumption on $\cA_\sigma$ implies that $\ker y\subset N_{-}$ and that $N_{-}\subset \im y$, since $y$ interchanges the eigenspaces. Therefore, $N = N_+\oplus yN_+ =\pi^* N_+$, so choosing the eigenspace $N_+$ is the inverse to the pullback functor with the equivariant structure. It is straightforward to check that morphisms of $R$-modules are in bijection (via the pullback) with equivariant morphisms of $S$-modules.

It only remains to show that for an equivariant module $N$, the condition $Li_Y^*\cA_\sigma = \Id$ is equivalent to the condition that $\ov\sigma$ acts as $1$ on $N/yN$ and as $-1$ on $y^{-1}(0)\subseteq N$. Using the presentation $S=R[y]/(y^2-x)$, we see that $I_Y =\langle  g^\sigma-g\rangle = yS$. A direct computation using the resolution $S\overset{y}\to S$ shows that $N/yN\cong L^0i_Y^*N$ and $y^{-1}(0)\cong L^1i_Y^*N$, yet these isomorphisms do not necessarily commute with $\cA_\sigma$, as we will show.


We begin by constructing a free resolution of $N$ that carries a compatible equivariant structure. First, split $N$ into eigenspaces $N = N_{+}\oplus N_-$ as before. Take generating sets of $N_+$ and $N_-$ as $R$-modules and consider the free $S$-module generated by the union, which we will write $F_0=F_0^+\oplus F_0^-$ ($F_0^{\pm}$ is generated by a generating set of $N_{\pm}$). We have the surjection $d_0:F_0^+\oplus F_0^-\to N$, and its pullback $\sigma^*F_0^+\oplus \sigma^*F_0^-\to \sigma^*N$. Next we extend the equivariant structure to $F_0$: For $e$ a basis element of $F_0^{\pm}$, we let $\cA_\sigma(e) = \pm \sigma^* e$. This ensures that we have the rightmost commutative square in the following diagram:


\[
\begin{tikzcd}[column sep = 2 em, ampersand replacement = \&]
F_2 \arrow[d,"\cA_\sigma"]\arrow[r,"d_2"]\&
F_1 \arrow[d,"\cA_\sigma"]\arrow[r,"d_1"]\&
F_0 \arrow[d,"\cA_\sigma"]\arrow[r,"d_0"]\& N\arrow[d,"\cA_\sigma"]\arrow[r]\& 0 \\
\sigma^*F_2\arrow[r,"\sigma^*d_2"]\&
\sigma^*F_1\arrow[r,"\sigma^*d_1"]\&
\sigma^*F_0\arrow[r,"\sigma^*d_0"]\& \sigma^*N \arrow[r]\& 0.
\end{tikzcd}
\]

Now, let $K_0=\ker d_0$. Notice that $\cA_\sigma K_0 = \sigma^*K_0$, so $K_0$ inherits the equivariant structure. Thus, we can iterate the process to obtain the beginning of a free resolution $F_2\overset{d_2}\to F_1\overset{d_1}\to F_0\overset{d_0}\to N\to 0$ where every term is equivariant and the above diagram is commutative.


Let us write $i=i_Y$. $Li^*N$ is represented by the complex $i^*F_\bullet = \cdots \to i^*F_2\overset{i^*d_2}\to i^*F_1\overset{i^*d_1}\to i^*F_0$ (the quasiisomorphism $Li^*F_\bullet\cong Li^*N$ is induced by the map $d_0:F_0\to N$), and $i^*M = M/yM$. We note that

\[
L^0i^*_YN\cong H^0(i^*F_\bullet) =\coker(i^*F_1\to i^*F_0)= \frac{F_0}{yF_0 + d_1F_1} \overset{d_0}{\underset\cong\longrightarrow} \frac{N}{yN}.
\]
The map $d_0$ commutes with $\ov \sigma$: therefore if $\ov\sigma$ acts as the identity on one side, it will do so in the other, as desired. For $L^1i^*N$, we note the following:
\[
L^1i^*_YN\cong H^{-1}(i^*F_\bullet) =\frac{\ker (i^*F_1\to i^*F_0 )}{\im (i^*F_2\to i^*F_1 )} = \frac{F_1\cap d_1^{-1}(yF_0)}{yF_1 + d_2F_2}= \frac{F_1\cap d_1^{-1}(yF_0)}{yF_1 + d_1^{-1}(0)} \overset{d_1}{\underset\cong\longrightarrow} \frac{d_1F_1\cap yF_0}{yd_1F_1}.\]
It is straightforward to check that $d_1$ induces an isomorphism. As before, $d_1$ commutes with $\ov\sigma$. Now, we distinguish two cases. Suppose first that $y^2\neq 0$. Since $R$ is a domain, $y$ is not a zero divisor. $F_0$ is free, so for any submodule $F'\subseteq F_0$, $y^{-1}(yF')=F'$, since $F_0$ is free. Therefore, $y$ induces an isomorphism:
\[
\frac{d_1F_1\cap yF_0}{yd_1F_1} \overset{y}{\underset\cong\longleftarrow} \frac{y^{-1}d_1F_1}{d_1F_1} .
\]
Notice that this map does not commute with $\ov \sigma$, but rather $y\circ \ov\sigma=-\ov \sigma\circ  y$. Therefore, if the action of $\ov\sigma$ on $L^1i^*N$ is $1$, the action on the right hand side is given by $-1$. Finally, notice that $d_0$ maps $y^{-1}d_1F_1/d_1F_1$ isomorphically into $y^{-1}(0)\subseteq N$, and that $d_0$ commutes with $\ov\sigma$. Notice that since $\sigma=\Id$ on $Y$, $\ov \sigma=\sigma^*\circ \cA_\sigma=\cA_\sigma$. This shows what we wished: if $\ov \sigma$ acts as $1$ on $N/yN$ and as $-1$ on $y^{-1}(0)\subset N$, then $\cA_\sigma$ acts as the identity on $Li^*N$.

Let us consider the case where $y^2=0$.





\end{proof}

\begin{proof}[Proof of Proposition~\ref{prop:EllipticModulesAreDihedral}]

Consider an elliptic module $M$, with $\cA:\pi_1^*M\to \pi_2^*M$. Lemma~\ref{lem:z2-descent} yields an equivariant sheaf structure on $\pi_1^*M$, $\cA_{\sigma_1}:\pi_1^*M\to \sigma_1^*\pi_1^*M$, and $\cA_{\sigma_1}$ is the identity at the ramification points. Making $\cA_\sigma=\cA$, we obtain a $G$-equivariant structure: the relations on $G$ are generated by $\sigma^2=\sigma_1^2=\Id$, and indeed $\sigma^*\cA\circ \cA_\sigma = \Id$.

Let us now go in the opposite direction. Let $\wt M$ be an equivariant sheaf on $E$ as in the statement. Lemma~\ref{lem:z2-descent} shows that there's a unique $M\in \QCoh (\P^1)$ such that $\wt M =\pi_1^*M$ with the induced $\sigma_1$-equivariant structure. Further, $\cA_{\sigma}$ induces an elliptic module structure on $M$.

It is straightforward to check that the constructions are functorial given that Lemma~\ref{lem:z2-descent} gives a functor, and that they are mutually inverse.

\end{proof}

Proposition~\ref{prop:FlatEllipticModulesAreDihedral-normalization} requires some background. If $E$ is singular and reduced, then the results of \cite{Ferrand} allow us to relate quasicoherent sheaves on $E$ with sheaves on its normalization $\wt E$. These results require flatness at the singular points, so we cannot have an equivalence (see Remark~\ref{rem:flatnessIsNeeded} for an example). However, we do have an equivalence between the full subcategories of flat sheaves in the equivariant setting, analogously to the theorem in loc. cit. We will recall it here for convenience.

This theorem describes the relation between modules over a fiber product of rings $B\times_{B'} A'$ and modules over $B$, $B'$ and $A'$. We reproduce the statement and the constructions here for convenience. Start with a Cartesian square of rings, and the corresponding commutative square of pullbacks (i.e. tensors):
\[
\begin{tikzcd}[column sep = 3 em, ampersand replacement = \&]
B\times_{B'} A' \arrow[r]\arrow[d] \&  A\arrow[d]
\& \RMod(B\times_{B'} A') \arrow[r]\arrow[d] \&  \RMod(A')\arrow[d] \\
B \arrow[r] \& B' \& \RMod(B) \arrow[r] \& \RMod(B') \& .
\end{tikzcd}
\]
The diagram on the right hand side induces a functor $T:\RMod(B\times_{B'} A')\to \RMod(B)\times_{\RMod(B')} \RMod(A')$, which concretely is given by
\[
T (M) = \left(B\otimes M, A'\otimes M, \cong \right).
\]
Recall that $\RMod(B)\times_{\RMod(B')} \RMod(A')$ is the category of triples consisting of a $B$-module $N_B$, an $A'$-module $M_{A'}$ and an isomorphism $\phi:B'\otimes N_B\cong B'\otimes N_{A'}$. In the definition of $T(M)$, this isomorphism is the canonical one. Ferrand constructs a right adjoint $S$ to $T$, defined as follows: an object $N = (N_B,N_{A'},\phi)$ is mapped to 
\[
S(N) = \{(n_B,n_{A'})\in N_B\times N_{A'}: \phi(1\otimes n_B) = 1\otimes n_{A'} \}.
\]
$S$ is defined on morphisms in the obvious way. Th\'eor\`eme 2.2 in \cite{Ferrand} includes the following statement.
\begin{theorem}[Ferrand]\label{thm:Ferrand}
For $A',B',B,S,T$ as above, assume that $A'\to B'$ is surjective. Then $S$ and $T$ are inverse equivalences between the full subcategories of consisting of flat modules.
\end{theorem}

\begin{proof}[Proof of Proposition~\ref{prop:FlatEllipticModulesAreDihedral-normalization}]


Let us start by showing that we are in the right situation to apply Theorem~\ref{thm:Ferrand}. Rings will be replaced by schemes affine over $E$, and analogous statements hold simply because modules and pullbacks are preserved by localization.

Let $\sigma_2=\sigma\sigma_1\sigma$ be the deck involution for $\pi_2$. Let $\wt X$ be the subscheme of $\wt E$ given as the fixed subscheme of $\sigma_1\sigma_2$. This is the subscheme cut out by the ideal sheaf $I_{\wt X} = \langle f-f^{\sigma_1\sigma_2}:f\in \O_{\wt E}\rangle =\langle f^{\sigma_1}-f^{\sigma_2}:f\in \O_{\wt E}\rangle $. Letting $X=\pi(\wt X)$, we have a commutative square:
\begin{equation}\label{eq:pinching}
\begin{tikzcd}[column sep = 3 em, ampersand replacement = \&]
\wt X\arrow[r,"i_{\wt X}"]\arrow[d,"\pi"] \&  \wt E\arrow[d,"\pi"] \\
X \arrow[r,"i_{ X}"] \& E .
\end{tikzcd}
\end{equation}

\begin{lemma}\label{lem:pinching}
With the notation above, $\O_E = \pi_*\O_{\wt E}\times_{i_{X*}\pi_*\O_{\wt X}}i_{X*}\O_X$. Further, $X$ is the (affine scheme) quotient of $\wt X$ by the action of $\sigma_1$, so $\pi_1$ induces an isomorphism between $X$ and its image. The support of $X$ is exactly $Z$, the singular set of $E$. Here we assume that the field $k$ is perfect and not of characteristic $2$.
\end{lemma}

\begin{proof}


Each of the two maps $\pi_i\circ \pi:\wt E\to \P^1$ is a Galois ramified cover with Galois group $\langle \sigma_i\rangle = \Z/2\Z$, so it identifies $\O_{\P^1}$ with $(\pi_{i*}\pi_{*}\O_{\wt E})^{\sigma_i}$, where the notation $R^{\sigma_i}$ denotes $\{f\in R: f^{\sigma_i}=f\}$. Since $E$ is the image of $\wt E$ in $\P^1\times \P^1$, $\O_E$ is generated by functions on each of the $\P^1$ factors. Our first claim is that $\pi_* \O_{\wt E}^{\sigma_1}\otimes_k \pi_* \O_{\wt E}^{\sigma_2}$ generates $\O_E$. This statement must be understood in the following sense: there is a basis of open sets $U$ of $E$ such that $\O_E(U)$ is generated by $\sigma_1$-invariant functions in $\pi_*\O_{\wt E}(\pi^{-1}(U))$, together with $\sigma_2$-invariant functions in $\pi_*\O_{\wt E}(\pi^{-1}(U))$. In particular, if we say $f\in \O_{\wt E}(V)$ is $\sigma_i$-invariant, we mean that $f^{\sigma_i}$ is regular on $V$ as well. Further, we make the same claim about $\pi_1{}_*\O_E$: we will show that there is a basis of open sets $U$ of $\P^1$ such that $(\pi_1{}_*\O_E)(U)$ is generated by $\sigma_1$-invariant functions and $\sigma_2$-invariant functions in $(\pi_1\pi)_*\O_{\wt E}(U)$. Note that all the rings of regular functions we mention can be thought of as contained in the ring of rational functions of $\wt E$ (recall that $\wt E$ might be disconnected, in which case its ring of rational functions is a sum of fields), so we can talk about containments and generation.

First, choose a basis of open sets of $E$ of the form $V = E\cap (U_1\times U_2)$, where $U_i\subseteq \P^1$ are affine open subschemes. The ring $\O_E(V)$ is generated by $\pi_i^{-1}{\O_{\P^1}}(U_i)=(\pi_*\O_{\wt E}((\pi_i\pi)^{-1}(U_i)))^{\sigma_i}$, for $i=1,2$, 
where we can think of all the rings as contained in the ring of rational functions of $\wt E$. Since $\pi_i^{-1}(U_i)\supseteq V$, $\O_{\wt E}((\pi_i\pi)^{-1}(U_i))\subseteq \O_{\wt E}(\pi^{-1}(V))$, so we have the desired statement on $E$: $\O_E(V)=\pi_*\O_{\wt E}(\pi^{-1}(V))^{\sigma_1}\cdot \pi_*\O_{\wt E}\pi^{-1}(V))^{\sigma_2}$. Let us see what happens on $\P^1$: suppose we have an open set as above, $E \cap (U_1\times U_2)$, and consider any open $U\subseteq \P^1$ such that $\pi_1^{-1}(U)\subseteq U_1\times U_2$. In this case, we have the simple observation that $\pi_1^{-1}(U) = E \cap (U\times U_2)$, so the reasoning above applies, and therefore $\O_E(\pi_1^{-1}U)$ is generated by $ (\pi_*\O_{\wt E}((\pi_1\pi)^{-1}(U)))^{\sigma_1}$ and $(\pi_*\O_{\wt E}((\pi_2\pi)^{-1}(U_2)))^{\sigma_2}$. By assumption, $\pi_1^{-1}(U)\subseteq \pi_2^{-1}(U_2)$, so
\[
(\pi_*\O_{\wt E}((\pi_2\pi)^{-1}(U_2)))^{\sigma_2}\subseteq (\pi_*\O_{\wt E}((\pi_1\pi)^{-1}(U)))^{\sigma_2}\subseteq \O_E(\pi_1^{-1}U).
\]
In particular, $\O_E(U)$ is generated by $(\pi_*\O_{\wt E}((\pi_1\pi)^{-1}(U)))^{\sigma_i}$ for $i=1$ and $i=2$. In particular, there is a basis for the topology on $\P^1$ over which the equation $\pi_{1*}\O_{E} = \pi_{1*}\pi_*\O_{\wt E}^{\sigma_1}\cdot \pi_{1*}\pi_*\O_{\wt E}^{\sigma_2}$ holds.

All four maps in the Diagram~(\ref{eq:pinching}) are affine, as is the map $\pi_1:E\to \P^1$. We will slightly abuse notation and use $\O_{\wt X},\O_X,\O_{\wt E},\O_E$ to refer to their pushforwards to $\P^1$ by the map $\pi_1$, taking advantage of the fact that schemes affine over $\P^1$ are equivalent to quasicoherent sheaves of $\O_{\P^1}$-algebras. Then, the statement we are trying to prove can be written $\O_E = \O_{\wt E}\times_{\O_{\wt X}} \O_X$. We will think of quasicoherent sheaves on a scheme $\Xi$ affine over $\P^1$ as sheaves of modules over $\O_{\Xi}$. Further, the discussion above shows that we may think of $\pi_*\O_E$ as $\pi_1{}_*\pi_*\O_{\wt E}^{\sigma_1}\cdot \pi_2{}_*\pi_*\O_{\wt E}^{\sigma_2}$, which we will just abbreviate as $\O_{\wt E}^{\sigma_1}\O_{\wt E}^{\sigma_2}$. We have the following diagram:
\[
\begin{tikzcd}[column sep = 5 em, ampersand replacement = \&]
\O_E = \O_{\wt E}^{\sigma_1}\O_{\wt E}^{\sigma_2} \arrow[r,hook]\arrow[d,"",two heads]\& \O_{\wt E}\arrow[d,"",two heads] \\
\frac{\O_{E}}{I_{\wt X}\cap \O_E} \arrow[r,hook,""]\& \O_{\wt X} .
\end{tikzcd}
\]
We claim it is Cartesian, by first showing that $I_{\wt X}\subset \O_E$: this is due to the fact that generators of $I_{\wt X}$ (on some small enough open set) can be written as $f-f^{\sigma_1\sigma_2} = (f + f^{\sigma_1})-(f^{\sigma_1} +f^{\sigma_1\sigma_2})\in \O_{\wt E}^{\sigma_1} + \O_{\wt E}^{\sigma_2}\subset \O_E$. Now, $\O_E$ is contained in the fiber product $\frac{\O_{E}}{I_{\wt X}\cap \O_E}\times_{\O_{\wt X}} \O_{\wt E}$, so we need to show the other containment: a local section in the fiber product is $s\in \O_{\wt E}$ such that $s + I_{\wt X}\in \O_E+I_{\wt X}$. Since $I_{\wt X}\subset \O_E$, it follows that $s\in \O_E$.

We have the desired cartesian square of sheaves of rings. Notice that $X=\pi(\wt X) = \Spec \O_E/I_{\wt X}$. Finally, to show that $X=\wt X/\langle \sigma_1\rangle$, we need to show that $\left({\O_{\wt E}/I_{\wt X}} \right)^{\sigma_1} = {\O_E}/{{I_{\wt X}}}$. First, $ {\O_E}/{{I_{\wt X}}}$ is generated as a sheaf of rings by $\O_{\wt E}^{\sigma_1}$ and $\O_{\wt E}^{\sigma_2}$, so in order to show that $\left({\O_{\wt E}/I_{\wt X}} \right)^{\sigma_1} \supseteq {\O_E}/{{I_{\wt X}}}$ we only need to check that $\O_{\wt E}^{\sigma_2}\subseteq \O_{\wt E}^{\sigma_1}+I_{\wt X}$. An element $f\in \O_{\wt E}^{\sigma_2}$ can be written as

\[
f = \frac{f+f^{\sigma_1}}{2} + \frac{f-f^{\sigma_1}}{2} = \frac{f+f^{\sigma_1}}{2} + \frac{f-f^{\sigma_2\sigma_1}}{2} \in \O_{\wt E}^{\sigma_1} + I_{\wt X}.
\]
For the other containment, let $f+I_{\wt X}\in \left({\O_{\wt E}/I_{\wt X}} \right)^{\sigma_1}$, i.e. suppose $f^{\sigma_1}-f=g\in I_{\wt X}$. Then $g^{\sigma_1} = -g$, $f + g/2\in \O_{\wt E}^{\sigma_1}\subset \O_E$ and $f +I_{\wt X} = f + g/2 + I_{\wt X}$, showing the desired containment.

Finally, let us show that the points of $X$ are those where $E$ is singular. First note that $I_{\wt X}$ is contained in the conductor of $\O_E\subseteq \O_{\wt E}$, since $I_{\wt X}\O_{\wt E} = I_{\wt X}\subseteq \O_E$, just because $I_{\wt X}$ is an ideal of $\O_{\wt E}$. Since the conductor is supported on the singular locus of $E$ (i.e. the points where $\pi$ is not an isomorphism), it follows that $\wt X$ contains $\pi^{-1}(Z)$.


For the other containment, suppose $p$ is a closed point in $\wt X$, i.e. $\sigma_1p=\sigma_2p$. There are two possible situations, depending on whether $\sigma_1p=p$. Start by assuming that $\sigma_1p\neq p$. In this case, for $i=1,2$, $\pi_i\pi(p) =  \pi_i\pi(\sigma_i p) =  \pi_i\pi(\sigma_{3-i} p)$, which implies that the map $(\pi_1\pi,\pi_2\pi):\wt E\to E\subset \P^1\times \P^1$ identifies $p$ and $\sigma_1p$. Therefore, $\pi$ is not an isomorphism around $p$, so the stalk of $E$ at $\pi(p)$ is not normal, hence $\pi(p)$ is singular.

Assume now that $\sigma_1p=p$, and let $m\subset \O_{\wt E}$ be the corresponding maximal ideal. Suppose $\sigma_1$ does not act as the identity on $\O_{\wt E}/m$. Then 
$(\O_{\wt E}/m)^{\sigma_1}\subsetneq \O_{\wt E}/m$. We have already seen that $\O_E/I_{\wt X} = (\O_{\wt E}/I_{\wt X} )^{\sigma_1}$, and therefore
\[
\frac{\O_E}{m\cap \O_E}\overset{I_{\wt X}\subset m\cap \O_E}=
\frac{\O_E/I_{\wt X}}{(m\cap \O_E)/I_{\wt X}}=\frac{(\O_{\wt E}/I_{\wt X} )^{\sigma_1}}{(m\cap \O_E)/I_{\wt X}}\subseteq \left(\frac{\O_{\wt E}}{m }\right)^{\sigma_1}\subsetneq \frac{\O_{\wt E}}{m}.
\]
So, as before, $\pi$ is not an isomorphism around $p$, so $\pi(p)$ is singular.

Lastly, suppose $\sigma_1$ acts as the identity on $\O_{\wt E}/m$. Then $\sigma_1$ acts linearly on $m/m^2$, which is a one dimensional $\O_{\wt E}/m$-vector space. Since $\sigma_1$ is an involution, it acts as $-1$ or as $1$. Suppose it acts as $1$: then for a generator $f$ of $m$, we have that $f^{\sigma_1} \in f+m^2$. Therefore, for any $n$, $ (f^n)^{\sigma_1} \in f^n + f^{n-1}m^2 = f^n + m^{n+1}$, so $\sigma_1$ acts as the identity on the completion of $\O_{\wt E}/m$, so it acts as the identity on the connected component of $\wt E$ containing $m$. Since we are assuming that $E$ is reduced, this cannot happen. Therefore, $\sigma_1$ acts as $(-1)$ on $m/m^2$.

We are left with the situation where $\sigma_1$ acts as the identity on $\O_{\wt E}/m$ and as $(-1)$ on $m/m^2$, and so does $\sigma_2$, since $p\in \wt X$, the subscheme where $\sigma_1=\sigma_2$. 
Let us prove that $I_{\wt X}\subset m^2$, i.e. that $\sigma_1 = \sigma_2\mod m^2$. The map $\sigma_1-\sigma_2$ is a $k$-linear derivation of $\O_{\wt E}/m$ with values in $m/m^2$: first of all, if $a\in m$, then $a^{\sigma_1} \equiv a^{\sigma_2} \equiv  -a \mod m^2$, and $\sigma_1=\sigma_2=1$ when they act on $\O_{\wt E}/m$, so it is a $k$-linear map as desired. Notice further that $m/m^2\cong \O_{\wt E}/m$ as $\O_{\wt E}/m$-vector spaces. Finally, we can check it is indeed a derivation: for any $a,b\in \O_{\wt E}/m$ and any lifts to $\O_{\wt E}/m^2$, we have that
\[
(ab)^{\sigma_1-\sigma_2} -a(b^{\sigma_1-\sigma_2}) -b(a^{\sigma_1-\sigma_2}) = (a-a^{\sigma_1})(b-b^{\sigma_1})-(a-a^{\sigma_2})(b-b^{\sigma_2})\in m^2+m^2 = m^2.\]
Finally, since $k$ is perfect, the finite field extension $k\subseteq \O_{\wt E}/m$ is separable, and therefore the only $k$-linear derivation of $\O_{\wt E}/m$ is $0$, so $\sigma_1=\sigma_2\mod m^2$ as desired.

Therefore, $I_{\wt X}\subseteq m^2$, and $(\O_{\wt E}/m^2)^{\sigma_1}\subsetneq \O_{\wt E}/m^2$. As before, we have that
\[
\frac{\O_E}{m^2\cap \O_E}\overset{I_{\wt X}\subset m^2\cap \O_E}=
\frac{\O_E/I_{\wt X}}{(m^2\cap \O_E)/I_{\wt X}}=\frac{(\O_{\wt E}/I_{\wt X} )^{\sigma_1}}{(m^2\cap \O_E)/I_{\wt X}}\subseteq \left(\frac{\O_{\wt E}}{m^2 }\right)^{\sigma_1}\subsetneq \frac{\O_{\wt E}}{m^2}.
\]
Therefore, $\pi$ is not an isomorphism around $p$, so $p$ is a singular point.

\end{proof}

Proposition~\ref{prop:EllipticModulesAreDihedral} shows that the category of elliptic modules is equivalent to the category of $G$-equivariant sheaves on $E$ with the condition that $Li^*_Y\cA_{\sigma_1} = \Id$. A sheaf on $\P^1$ is flat at $\pi_1(Z)$ if and only if it has no torsion supported on $\pi_1(Z)$, equivalently, if and only if its pullback to $E$ has no torsion supported on $E$. Therefore, the equivalence into Proposition~\ref{prop:EllipticModulesAreDihedral} restricts to an equivalence between the desired subcategory of elliptic modules and the category of $G$-equivariant sheaves on $E$ which are flat at $Z$ and such that $Li^*_Y\cA_{\sigma_1} = \Id$.

Let us start by showing how $\pi^*$ maps equivariant modules to equivariant modules. Let $M\in \Mod(E)$: we have the maps $\pi^*\cA_{\sigma}:\pi^*M\to \pi^*\sigma^*M = \sigma^*\pi^*M$, and $\sigma^* (\pi^*\cA) = \pi^*\sigma^* \cA = (\pi^*\cA)^{-1}$. Similarly, we have $\pi^*\cA_{\sigma_1}$ and both maps together make $\pi^*M$ $G$-equivariant. If $M$ is flat at $Z$, $\pi^*M$ is flat at $\pi^{-1}(Z)$. Further, suppose $Li^*_Y\cA_{\sigma_1}= \Id$. Then, considering the restriction $\pi:\wt Y\to Y$, we have that $Li^*_{\wt Y}L\pi^*\cA_{\sigma_1} = L\pi^*Li^*_Y\cA_{\sigma_1} = \Id$. Now, note that on a neighborhood of the points of $Y\setminus Z$, $\pi$ is an isomorphism, and therefore $L\pi^*=\pi^*$. On the other hand, on a neighborhood of the points of $Y\cap Z$, we are assuming that $\pi_1^*M$ is flat, and therefore $L\pi^*\cA_{\sigma_1}=\pi^*\cA_{\sigma_1}$. Therefore, $Li^*_{\wt Y}\pi^*\cA_{\sigma_1} =\Id$ as desired. This provides a functor going one way.

Let us now construct the inverse to $\pi^*$. Given Lemma \ref{lem:pinching}, we are in the situation where Theorem~\ref{thm:Ferrand} applies. We have the adjoint pair of the descent functor $S:\QCoh(\wt E)\times_{\QCoh(\wt X)} \QCoh(X)\to \QCoh(E)$ and its right adjoint $T$ given by pullbacks to $\wt E$ and $X$. $S$ is given on objects by mapping a triple $N_{\wt E}\in \QCoh(\wt E)$, $N_X\in \QCoh(X)$ and $\phi: i_{\wt X}^*N_{\wt E}\cong \pi^*N_X$ to
\[
S(N_{\wt E},\phi,N_X) = \{(\pi_*s_{\wt E},i_{\wt X}{}_*s_X)\in \pi_*N_{\wt E}\times i_X{}_*N_X: \phi(i_{\wt X}^*s_{\wt E}) = \pi^*s_X \}.
\]

Consider $\wt M\in \Mod(\wt E)$ satisfying the hypotheses in the statement. From $\wt M$ we construct an object in $\QCoh(\wt E)\times_{\QCoh(\wt X)} \QCoh(X)$: The $\langle \sigma_1 \rangle$-equivariant structure $\cA_{\sigma_1}$ satisfies the hypothesis of Lemma~\ref{lem:z2-descent}, so there is a sheaf $M\in \QCoh(\P^1)$ such that $\pi^*\pi_1^*M = \wt M$ with this equivariant structure. Take $T(\pi_1^*M) =(\wt M,i_X^*\pi_1^*M)$ to be the desired object. Since $\wt M$ is flat at $\pi^{-1}(Z)$, $M$ is flat (i.e. torsion-free) at $\pi(Z)$, and $\pi_1^*M$ is flat at $Z$. Equivalently, by the local criterion for flatness, $\pi_1^*M$ is flat at $X$: Lemma~\ref{lem:pinching} shows $Z$ and $X$ have the same support. Theorem~\ref{thm:Ferrand} then implies that $\pi_1^*M$ and $T(\pi_1^*M)$ are in the categories on which $T$ and $S$ are inverse equivalences, so in particular we have the natural isomorphism $\pi_1^*M \to S(T(\pi_1^*M))$.

To give $M$ the structure of an elliptic module, we need to construct a $\sigma$-equivariant structure. The $\sigma$-equivariant structure of $\wt M$ can be enhanced to one on $T(\pi_1^*M)$, by simply restricting $\cA_{\sigma}$ to $\wt X$ (note that $\wt X$ is $G$-invariant). Now, we simply take $S(\cA_{\sigma}):S(\wt M)\to S(\sigma^*M)$. From the definition of $S$ above, it is clear that $S(\sigma^*M)$ is naturally isomorphic to $\sigma^*S(M)$, providing the desired equivariant structure, since indeed $\sigma^*S(\cA_{\sigma})\circ S(\cA_{\sigma}) =\Id$.

Let us show that this construction is a functor: a morphism of $G$-equivariant sheaves $\wt f:\wt M\to \wt N$ is mapped to a morphism $f:M\to N$ of sheaves on $\P^1$ since Lemma~\ref{lem:z2-descent} provides a functor (and $\sigma_1$-equivariant maps descend to $\P^1$). Further, $\pi_1^*f$ will be $\sigma_1$-equivariant.

It remains to show that if we start with a morphism of $G$-equivariant sheaves, then $\pi_1^*f$ will be $\sigma$-equivariant. Suppose that $\sigma^*\wt f\circ \cA_\sigma= \cA_\sigma\circ \wt f$. We have the morphism $T(\pi_1^*f):T(\pi_1^*M)\to T(\pi_1^*N)$, and we want to show that $\sigma^*T(\pi_1^*f)\circ \cA_\sigma= \cA_\sigma\circ T(\pi_1^*f)$. For this, we need the identity to hold on $\wt E$ and on $X$. It holds on $\wt E$ by hypothesis, and in order to hold on $X$, must have that
\[
i_X^*\pi_2^* f \circ i_X^*\cA_\sigma = i_X^*\sigma^*\cA_\sigma \circ i_X^*\pi_1^*f.
\]
It is true that $\pi^*$ applied to the above equation holds, since $\pi^*i_X^*\pi_1^*f=i_{\wt X}^*\wt f$. Now, $\pi:\wt X\to X$ is the restriction of $\pi_1\pi:\wt E\to \P^1$ to the preimage of $\pi_1(X)\cong X$, so it is a faithfully flat map. Therefore, $\pi^*$ from $X$ to $\wt X$ is a faithful functor, and the above equation holds since it holds after taking $\pi^*$. Now we have that $\sigma^*T(\pi_1^*f)\circ \cA_\sigma= \cA_\sigma\circ T(\pi_1^*f)$. Applying $S$ to this equation we have the desired equivariance of $f$.

Given that $S$ and $T$ are mutually inverse, it is straightforward to check that the functors we have constructed are mutually inverse.

\end{proof}




\begin{remark}\label{rem:flatnessIsNeeded}
The condition of flatness at $Z$ is indeed necessary. Consider the following example: Let an affine open set of $E$ be cut out by the equation $(y-qx)(y-q^{-1}x)=0$, for some $q\in k^\times$ with $q^2\neq 1$. Then $\wt E$ is the disjoint union of two lines: $\wt E= \Spec k[t_1]  \times k[t_2]$, where $\pi$ is given by $\pi^*x= (t_1,t_2)$ and $\pi^*y= (qt_1,q^{-1}t_2)$. The dihedral group $G$ acts as follows:
\[
\begin{array}{rrcl|rrcl}
\sigma_1: & x & \leftrightarrow & x & \sigma: & x & \leftrightarrow & y \\
 & y & \leftrightarrow & (q+q^{-1})x-y &  & y & \leftrightarrow & x \\
 & (t_1,0) & \leftrightarrow & (0,t_2)  &  & (t_1,0) & \leftrightarrow & (0,q^{-1}t_2) \\
\end{array}
\]
We can consider the following $G$-equivariant sheaf on $\wt E$: let $\wt M = k[t_1]/(t_1) \times k[t_2]/(t_2)$, and let $s_i$ be a basis element for $k[t_i]/(t_i)$. Consider the following equivariant structure:\[
\cA_{\sigma_1}(s_i)=\sigma_1^*s_{2-i}; 
\cA_{\sigma}(s_i)=\sigma^*	s_{2-i}; \quad i =1,2
\]This equivariant structure satisfies the condition that $Li^*_{\wt Y}\cA_{\sigma_1} =\Id$ vacuously, since ${\wt Y}$ is empty. Indeed $\wt M$ descends to $M=k[x]/(x)$ on $\Spec k[x]$.

There is no $E$-elliptic module whose pullback is $\wt M$: its underlying sheaf on $\P^1$ would have to be $M=k[x]/(x)$. However, $\pi_1^*M \cong k[x,y]/((y-qx)(y-q^{-1}x),x) \cong  k[x,y]/(y^2,x) $ is not isomorphic to $\pi_2^*M\cong  k[x,y]/(y,x^2)$. Therefore, $M$ supports no elliptic module structure.
\end{remark}

\begin{remark}\label{rem:flatnessIsNeeded-notFull}
The functor from $E$-elliptic modules to $G$-equivariant sheaves on $\wt E$ constructed above from $\pi^*\pi_1^*$ is defined for any elliptic module, without the flatness assumption. The functor defined this way on the whole category of elliptic modules is faithful, but not full in general. Consider two elliptic modules $M$ and $N$, with their corresponding elliptic structures which we will denote by $\cA$ in both cases.

Lemma~\ref{lem:z2-descent} ensures that $\pi^*\pi_1^*$ is a bijection between morphisms of sheaves from $M$ to $N$ and morphisms of $\Z/2\Z\langle \sigma_1\rangle$-equivariant sheaves from $\pi^*\pi_1^*M$ to $\pi^*\pi_1^*N$. Therefore, the map $\pi^*\pi_1^*:\Hom_{\EMod}(M,N)\to \Hom_G(\pi^*\pi_1^*M,\pi^*\pi_1^*N)$ is injective, since it is the restriction of the bijection $\pi^*\pi_1^*:\Hom_{\O_{\P^1}}(M,N)\overset{\sim}\to \linebreak \Hom_{\Z/2\Z}(\pi^*\pi_1^*M,\pi^*\pi_1^*N)$.

Let us now show by example that the functor is not full. Consider the curve $E$ from Remark~\ref{rem:flatnessIsNeeded}, with $q$ a primitive cubic root of unity, so an affine open set of $E$ is cut out by the equation $y^2+xy+x^2=0$. We will construct two nonisomorphic elliptic modules $M_1,M_2$ whose pullbacks to $\wt E$ are isomorphic. For both modules, the underlying sheaf is the module $k[x]/(x^3)$. Let $s_i$ be the generator for $M_i$. We define the elliptic module structures by
\[
\cA_{1}\pi_1^*s_1 = \pi_2^*s_1;\quad\cA_{2}\pi_1^*s_2 = (1+x^2y)\pi_2^*s_2.
\]
When pulled back to $\wt E$, they both take the form $\cA:\pi^*\pi_1^*s_i\mapsto \pi^*\pi_2^*s_i$, so they become isomorphic by mapping $s_1$ to $s_2$. However, there are no nonzero maps from the elliptic module $M_1$ to $M_2$. Such a map would take the form $f(s_1) = (a_0+a_1x+a_2x^2)s_2$. The relation $\cA_2 \circ \pi_1^*f = \pi_2^*f\circ \cA_1$ amounts to
\[(a_0+a_1x+a_2x^2+a_0x^2y)\pi_2^*s_2=(a_0+a_1x+a_2x^2) (1+x^2y)\pi_2^*s_2=\cA_2\left(
(a_0+a_1x+a_2x^2)\pi_1^*s_2
\right)=
\cA_2\left(\pi_1^* (f(s_1))\right) = \]\[
=\cA_2\left(\pi_1^* f(\pi_1^*s_1))\right)=\pi_2^*f \left( \cA_1(\pi_1^*s_1)\right)=
\pi_2^*f\left(
\pi_2^*s_1\right) = (a_0+a_1y+a_2y^2)\pi_2^*s_2.
\]
The only solution to the equation $a_0+a_1x+a_2x^2+a_0x^2y\equiv a_0+a_1y+a_2y^2\mod (y^2+xy+x^2,x^3)$ corresponds to the zero morphism.

\end{remark}

\begin{remark}\label{rem:possibleEllEq}
With Proposition~\ref{prop:EllipticModulesAreDihedral} in mind, it seems that there are several reasonable definitions for elliptic modules. One could consider the whole category of $G$-equivariant sheaves on $E$, which as explained in said Proposition contains $\EMod$ as a full subcategory. Alternatively, one could force $\sigma$ and $\sigma_1$ to play symmetric roles by requiring that $\cA_\sigma$ act as the identity on the fixed locus of $\sigma$, and considering this full subcategory of the one we are calling $\EMod$ in this paper.

Also notice that there are two very different behaviors depending on whether $\sigma_1\sigma$ has finite order. If $(\sigma_1\sigma)^n=\Id_E$, then the composition $(\ov {\sigma_1}\ov \sigma)^n$ is an automorphism of $\pi_1^*M$. An interesting full subcategory of elliptic modules is the full subcategory of modules for which this automorphism is the identity.  In other words, one might consider sheaves equivariant for a finite dihedral group, rather than the infinite dihedral group.


\end{remark}

\subsubsection{Application to elliptic equations}

In light of Proposition~\ref{prop:EllipticModulesAreDihedral}, we can apply Theorem~\ref{thm:fiberFinSt} to elliptic modules.

\begin{theorem}\label{thm:mainThm-elliptic}
Let $E$ and $G$ be as in Proposition~\ref{prop:EllipticModulesAreDihedral}. Let $p\in E$ be a closed point and let $E^*=E\setminus Gp$. For any scheme on which $\sigma_1$ acts, we will denote without ambiguity $Y$ as the fixed scheme of $\sigma_1$. Let $p
\in E$. Let $\Hol(E)^\circ$ (resp. $\Hol(E^*)^\circ$) be the full subcategory of $\Hol(E)$ (resp. $\Hol(E^*)$) consisting of modules for which $Li^*_Y\cA_{\sigma_1} = \Id$.

To define $\Mod(\pz)^\circ$, for every $g\in G$ we will let $Y_g=g^{-1}Y$ be the fixed scheme of $g^{-1}\sigma_1g$ intersected with the formal neighborhood of $p$, in particular $Y_g$ is empty unless $\sigma_1gp=gp$. Then, we let $\Mod(\pz)^\circ$ be the full subcategory of $\Mod(\pz)$ consisting of modules for which $Li^*_{Y_g} \cA_{g^{-1}\sigma_1g}=\Id$, for every $g\in G$. We are denoting the embedding of $Y_g$ into $\pz$ by $i_{Y_g}$.

Then the restriction of the functors $|_\pz$ and $|_\As$ induces an equivalence between $\Hol(E)^\circ$ and the fiber product $\Mod(\pz)^\circ\times_{\Mod(\pzz)} \Hol(E^*)^\circ$.
\end{theorem}

\begin{proof}
Clearly $|_\As$ maps $\Hol(E)^\circ$ into $\Mod(E^*)^\circ$. Also, $|_\pz$ maps $\Hol(E)^\circ$ into $\Mod(\pz)^\circ$: if $Li^*_Y\cA_{\sigma_1}=\Id$, then we use the following identity, which comes from applying Definition~\ref{def:equivariant} and the discussion thereafter:
\[
Li_g^*\cA_{g^{-1}\sigma_1g}=Li_1^* (g^{-1})^*\left(g^*\cA_{g^{-1}\sigma_1} \circ \cA_g \right)=
Li_1^*\left(\cA_{g^{-1}\sigma_1}\circ (g^{-1})^*\cA_g \right)=
Li_1^*\left(\sigma_1^*\cA_{g^{-1}}\circ \cA_{\sigma_1}\circ (g^{-1})^*\cA_g \right).
\]
Therefore, if $Li_1^*\cA_{\sigma_1}=\Id$, we have that
\[
Li_g^*\cA_{g^{-1}\sigma_1g}=Li_1^*\left(\sigma_1^*\cA_{g^{-1}}\circ (g^{-1})^*\cA_g \right) \overset{\sigma_1\circ i_1=i_1}= Li_1^*\left(\cA_{g^{-1}}\circ (g^{-1})^*\cA_g \right) = Li_1^*\cA_1 = \Id.
\]

Applying Theorem~\ref{thm:fiberFinSt}, $\Hol(E)\cong \Mod(\pz)\times_{\Mod(\pzz)} \Hol(E^*)$ contains $\Mod(\pz)^\circ\\\times_{\Mod(\pzz)} \Hol(E^*)^\circ$ as a full subcategory, which itself contains $\Hol(E)^\circ$ by the discussion above. It only remains to prove that $\Hol(E)^\circ\supseteq \Mod(\pz)^\circ\times_{\Mod(\pzz)} \Hol(E^*)^\circ$. Since we are dealing with full subcategories, we only need to check the containment of objects: we need to prove that for $M\in \Hol(E)$, if $Li^*_Y\cA_{\sigma_1}$ acts as the identity on both $Li^*_YM|_{E^*}$ and $Li^*_{Y_g}M|_\pz$ for every $g$, then $Li^*_Y\cA_{\sigma_1}$ is the identity on $Li^*_YM$ as well.

Let us show this: Let $K$ be the image of $Li^*_Y\cA_{\sigma_1}-\Id$. Since $(Li^*_Y\cA_{\sigma_1}-\Id)|_{E^*}=0$ and $|_{E^*}$ is an exact functor, $K$ is supported away from $E^*$ i.e. on $Gp$. Since $|_\pz$ is an exact functor, we also have that the formal fiber $K_p$ vanishes. We are left with the points $gp$ in the orbit of $p$. If $gp$ is fixed by $\sigma_1$, we use the equation above: $\Id=Li_g^*\cA_{g^{-1}\sigma_1g} = Li_1^*\left(\cA_{g^{-1}}\circ \cA_{\sigma_1}\circ (g^{-1})^*\cA_g\right)$, which implies that $Li_1^*\cA_{\sigma_1}=\Id$ at the stalk around $gp$ as well. Therefore, all the stalks of $K$ vanish, so indeed $M\in \Hol(E)^\circ$.
\end{proof}

\subsection{Relation to difference and differential equations on the line}\label{sec:elliptic-discrete}

Elliptic equations generalize discrete equations such as difference equations, i.e. sheaves equivariant under $z\mapsto z+1$, and $q$-equations, i.e. sheaves equivariant under $z\mapsto qz$, where $q\in k^\times$ is fixed (note that up to a change of coordinates on $\P^1$ these are all the automorphisms). This happens when the curve $E$ is reducible, in which case its components have degree $(1,1)$ (since they are not allowed to be fibers), and therefore each component is the graph $\Gamma_\tau $ of an automorphism $\tau$ of $\P^1$. Since $E$ is preserved by interchanging the coordinates there are two possibilities: either the components are interchanged, in which case they are the graph of an automorphism $\tau$ and its inverse (which must be different from $\tau$, so $\tau^2\neq 1$); or they are both preserved, in which case we have the graphs of two different involutions, one of which could possibly be the identity.

In the case where $E=\Gamma_\tau\cup \Gamma_{\tau^{-1}}$, elliptic equations are strongly related to $\tau$-equivariant sheaves on $\P^1$, which are difference equations if $\tau$ is $z\mapsto z+1$ and $q$-equations if $\tau$ is $z\mapsto qz$ (note that these are the only possibilities up to a change of coordinates). Away from the fixed points of $\tau$, the notions of an $E$-elliptic module and a $\Z\langle\tau\rangle$-equivariant sheaf are equivalent, and this equivalence can be extended over the special points for flat sheaves, as Proposition~\ref{prop:EllipticToDifference} shows.




Notice that the fixed \textbf{geometric} points of $\tau$ are the images of the singular geometric points of $\Gamma_{\tau}\cup \Gamma_{\tau^{-1}}$. In the situation where $E=\Gamma_{\tau_1}\cup \Gamma_{\tau_2}$, the singular geometric points are the preimages of the points $p$ for which $\tau_1p=\tau_2p$, or equivalently fixed geometric points of $\tau_1\tau_2$.

\begin{proposition}\label{prop:EllipticToDifference}
Let $k$ be perfect and not of characteristic $2$. Suppose $\tau\in \Aut(\P^1)$ is such that $\tau^2\neq 1$. Let $E = \Gamma_{\tau}\cup \Gamma_{\tau^{-1}}$ and let $Z$ be the fixed scheme of $\tau$.  Then the category of $\Z\langle\tau\rangle$-equivariant sheaves on $\P^1$ is equivalent to the category of equivariant sheaves on the curve $\wt E=\Gamma_{\tau}\sqcup \Gamma_{\tau^{-1}}$. Therefore, the following categories are equivalent.
\begin{enumerate}
\item $\tau$-equivariant sheaves on $\P^1$ which are flat at $Z$.
\item $E$-elliptic modules on the curve $E=\Gamma_{\tau}\cup \Gamma_{\tau^{-1}}$ which are flat at $Z$.
\end{enumerate}

Suppose we are given $\tau_1\neq \tau_2\in \Aut(\P^1)$ such that $\tau_j^2=\Id$, and $E=\Gamma_{\tau_1}\cup \Gamma_{\tau_2}$. Let $\wt G$ be the infinite dihedral group generated by $\tau_1$ and $\tau_2$, acting on $\P^1$ (the action is not necessarily faithful, for example if $\tau_1=\Id$). Let $Z$ be the fixed scheme of $\tau_1\tau_2$. Then the category of $\wt G$-equivariant sheaves on $\P^1$ is equivalent to the category of equivariant sheaves on the curve $\wt E=\Gamma_{\tau_1}\sqcup \Gamma_{\tau_2}$. Therefore, the following categories are equivalent.
\begin{enumerate}
\item $\wt G$-equivariant sheaves on $\P^1$ which are flat at $Z$.
\item $E$-elliptic modules on the curve $E=\Gamma_{\tau_1}\cup \Gamma_{\tau_2}$ which are flat at $Z$.
\end{enumerate}
\end{proposition}
\begin{proof}

Applying Proposition~\ref{prop:FlatEllipticModulesAreDihedral-normalization}, we have that elliptic modules which are flat at the singular points are equivalent to modules equivariant for the action of the dihedral group, and which are flat at the preimages of these singular points. The condition that $Li_{\wt Y}^*\cA_{\sigma_1}=\Id$ doesn't come into play, because in this case $\sigma_1$ acts freely on $\Gamma_\tau \sqcup \Gamma_{\tau^{-1}}$, since it interchanges the two components.

It remains to check that $G$-equivariant sheaves on $\Gamma_\tau\sqcup \Gamma_{\tau^{-1}}$ (which are flat at $\pi^{-1}(Z)$) are equivalent to $\tau$-equivariant sheaves on $\P^1$ (which are flat at $Z$). Given such an equivariant sheaf $M$ on $\P^1$, we may pull it back by the projection $\pi_1:\Gamma_\tau\sqcup \Gamma_{\tau^{-1}}\to \P^1$, and it automatically becomes $\frac{\Z}{2\Z}\langle \sigma_1\rangle$-equivariant (Lemma~\ref{lem:z2-descent}). The action of $\sigma$ is given by $\sigma|_{\Gamma_{\tau^{\pm 1}}} = \sigma_1\circ \tau^{\pm 1}:\Gamma_{\tau^{\pm 1}}\to \Gamma_{\tau^{\mp 1}}$, and therefore we must define
\[
\cA_{\sigma} = (\tau^{\pm 1})^*\cA_{\sigma_1} \circ \cA_{\tau^{\pm 1}} \quad \text{on }\Gamma_{\tau^{\pm 1}} .
\]
It is straightforward to check that indeed $\sigma^*\cA_{\sigma}\circ \cA_\sigma=\Id$, so $\pi_1^*M$ is $G$-equivariant. If $M$ is flat at $Z$, then $\pi_1^*M$ is flat at $\pi^{-1}(Z)$. Going back, if we start with $N$ on $\Gamma_\tau\sqcup \Gamma_{\tau^{-1}}$ which is $G$-equivariant, we can get a sheaf on $\P^1$ by taking $M=\pi_1{}_*(N|_{\Gamma_\tau})$. Then on $M$ we let $\cA_\tau = (\sigma^*\cA_{\sigma_1}\circ \cA_\sigma)|_{\Gamma_\tau}$. If $N$ is flat at $\pi^{-1}(Z)$, then $M$ is flat at $Z$. It is straightforward to check that these constructions are mutually inverse.

In the second situation, we proceed analogously: we must show that $G$-equivariant sheaves on $\Gamma_{\tau_1}\sqcup \Gamma_{\tau_2}$ are equivalent to $\wt G$-equivariant sheaves on $\P^1$, and that the flatness condition is preserved. As above, given a $\wt G$-equivariant sheaf $M$ on $\P^1$, we consider $\pi_1^*M$ as a $\Z/2\Z\langle \sigma_1\rangle$-equivariant sheaf. This time, the action of $\sigma$ on $\Gamma_{\tau_i}$ equals $\tau_i$, so we define $\cA_\sigma = \cA_{\tau_i}$ on $\Gamma_{\tau_i}$. As before, $\pi_1^*M$ becomes $G$-equivariant and it is flat at the fixed points of $\tau_2\tau_1$ if $\wt M$ was. The inverse of this functor is given as follows: starting with an equivariant sheaf $N$ on $\Gamma_{\tau_1}\sqcup \Gamma_{\tau_2}$, we let $M=\pi_1{}_*(N|_{\tau_1})$. The $\wt G$-equivariant structure is given by $\cA_{\tau_1} = \cA_\sigma|_{\Gamma_{\tau_1}}$, and $\ov{\tau_2}= \ov{\sigma_1}\circ\ov{\sigma}\circ\ov{\sigma_1}$, since the analogous relations hold for the action of $\wt G$ on $\P^1$. Again, flatness at the specified points is preserved and one can check that the constructions are mutual inverses.

\end{proof}

Recall that the flatness condition cannot be completely removed, as the example in Remark~\ref{rem:flatnessIsNeeded} shows.

Further, if the components of $E$ coincide so that $E$ becomes the double diagonal, then $E$-elliptic modules become strongly related to $D$-modules on $\P^1$. This is very similar to Grothendieck's definition of a connection, see \cite[I \S2]{deligne}.


\begin{proposition}\label{prop:elliptic-Dmod}
Let $\tau\in \Aut(\P^1)$ be such that $\tau^2=\Id$. Let $I$ be the ideal sheaf of the graph of $\tau$ in $\P^1\times \P^1$, and let $E$ be the subscheme cut out by $I^2$ and let $\Delta$ be the diagonal. Then $E$-elliptic modules are equivalent to the following:
\begin{itemize}
\item If $\tau=\Id$, elliptic modules are equivalent to ordered pairs of $D$-modules on $\P^1$, i.e. $\EMod\cong \DMod(\P^1) \oplus \DMod(\P^1)$. The full subcategory of elliptic modules such that $\cA|_\Delta=\Id$ is equivalent to $\DMod(\P^1)$.
\item If $\tau\neq \Id$, elliptic modules are equivalent to quasicoherent sheaves $M$ on $\P^1$ with two structures:
\begin{itemize}
\item A $\Z/2\Z$-equivariant structure $\cA_\tau:M\to \tau^*M$.
\item A connection $\nabla:M\to \Omega\otimes M$.
\end{itemize}
These two structures are compatible in the sense that $\tau^*\nabla\circ \cA_\tau = (\Id_\Omega\otimes \cA_\tau)\circ \nabla$. In other words, given $m\in M$, if we let $\nabla m=\sum \alpha_i\otimes m\in \Omega\otimes M$, then
we have that
\[
\nabla (\ov\tau m) =\sum \tau^*\alpha_i \otimes \ov\tau  m_i
\]
\end{itemize}
\end{proposition}
\begin{proof}

Let $\tau=1$ and consider an elliptic module $M$. Consider $\cA |_\Delta$, where $\Delta$ is the diagonal: since $\sigma|_\Delta = \Id$, $\cA|_\Delta$ is an endomorphism of $M$ whose square is the identity. $M$ then decomposes as the direct sum of its eigenspaces $M_{1}\oplus M_{-1}$. First of all, we claim that $M_{\pm 1}$ are $E$-submodules.

Consider the restriction $\cA:\pi_1^* M_1 \to \pi_2^* M_1\oplus \pi_2^* M_{-1}$. This map becomes the identity when restricted to $\Delta$, the diagonal, so its image is contained in $\pi_2^* M_1\oplus I_\Delta \pi_2^* M_{-1}$, where $I_\Delta$ is the ideal sheaf cutting out the diagonal. Consider a local section $m\in M_1$, and let $\cA (\pi_1^*m) = \pi_2^* m + m_1 + m_{-1}$, where $m_1\in I_\Delta\pi_2^* M_1$ and $m_{-1}\in I_\Delta\pi_2^*M_{-1}$. Notice that $\ov \sigma = \sigma^*\circ \cA$ acts as $\mp 1$ on $I_\Delta \pi_1^*M_{\pm 1}$: this sheaf is generated by elements of the form $(\pi_1^*f-\pi_2^*f)\pi_1^*n$, for $f\in \O_{\P^1}$ and $n\in M_{\pm 1}$. We have that
\begin{align}\label{eq:connection}
\begin{split}
\sigma^*\left(\cA((\pi_1^*f-\pi_2^*f)\pi_1^*n) \right) = 
\sigma^*\left((\pi_1^*f-\pi_2^*f)\cA(\pi_1^*n) \right) \in\sigma^*\left((\pi_1^*f-\pi_2^*f)(\pm\pi_2^*n+I_\Delta\pi_2^*M) \right) \overset{\mod I_\Delta^2}\equiv\\
\equiv  \sigma^*\left((\pi_1^*f-\pi_2^*f)(\pm\pi_2^*n) \right) = (\pi_2^*f-\pi_1^*f)(\pm \pi_1^*n) = \mp(\pi_1^*f-\pi_2^*f)\pi_1^*n
\end{split}
\end{align}
Now, given $\cA (\pi_1^*m) = \pi_2^* m + m_1 + m_{-1}$, let us write the equation $\sigma^*\cA (\cA(\pi_1^*m))=\pi_1^*m$:
\[
\pi_1^*m = (\sigma^*\cA)(\cA(\pi_1^*m)) = (\sigma^*\cA)(\pi_2^*m+m_1 + m_{-1}) =(\sigma^*\circ \cA\circ \sigma^*)(\pi_2^*m+m_1 + m_{-1}) \]\[= \sigma^* ( \cA(\pi_1^*m)) +\sigma^*(\cA(\sigma^* m_1+ \sigma^*m_{-1} ))= \sigma^* (\pi_2^*m + m_1+m_{-1} ) -\sigma^*m_1+\sigma^*m_{-1} = \pi_1^* m- 2m_{-1} 
.\]
This implies that $m_{-1}=0$, so $M_1$ is an $E$-submodule. The same computation shows that $M_{-1}$ is also an $E$-submodule.

The action of $\cA$ on $M_1$ is the same as a connection by the definition of Grothendieck, see for example \cite{deligne}. Similarly, the action of $-\cA$ on $M_{-1}$ is a connection. Therefore, elliptic modules consist of the direct sum of two $D$-modules. If we impose the condition that $\cA|_\Delta=\Id$, then $M_{-1}=0$, so we just obtain one $D$-module.

Now suppose that $\tau:\P^1\to \P^1$ is an involution, and let $E$ be the doubled graph of $\tau$, with its two projections $\pi_1,\pi_2:E\to \P^1$. Consider an $E$-elliptic module $M$. Let us write the second projection as $\pi_2=\tau\circ \pi_3$, so we have that $\pi_1\circ \sigma=\tau\circ\pi_3$. The elliptic module structure is an isomorphism $\cA:\pi_1^*M\to \pi_3^*\tau^*M$, such that $\sigma^*\cA = \cA^{-1}$. If we embed $E$ in $\P^1\times \P^1$ by $(\pi_1,\pi_3)$, it becomes the double diagonal. Let $\Delta$ be the diagonal of $\P^1\times \P^1$, embedded by $(\pi_1,\pi_3)$. Consider $\cA|_\Delta:M = \pi_1^*M |_\Delta\to  \pi_3^*\tau^* M|_\Delta = \tau^*M$, which gives $M$ a $\tau$-equivariant structure. Since $\sigma^*\cA\circ \cA = 1$, $M$ is $\Z/2\Z$-equivariant, and not just $\Z$-equivariant. Let $\cA_\tau = \cA|_\Delta:M\to \tau^*M$.

Consider the adjunction map $J_{\tau^*M}:\tau^*M\to \pi_3{}_*\pi_3^*\tau^*M$. Since $\pi_3{}_*\pi_3^*\tau^*M=\pi_1{}_*\pi_3^*\tau^*M$ as sheaves of groups, $J_{\tau^*M}$ can be seen as a $k$-linear map to $\pi_1{}_*\pi_3^*\tau^*M$. This is the map that to a section assigns its first order jet, and analogously we have $J=J_M:M\to\pi_1{}_*\pi_3^*M$. We define the following composition $\nabla$:
\[
\begin{tikzcd}[column sep = 5 em, ampersand replacement = \&]
-\nabla:M \arrow[r,"\cA- J_{\tau^*M}\circ \cA_\tau"]\& \pi_1{}_*\pi_3^*\tau^*M\arrow[r,"\pi_1{}_*\pi_3^*\tau^*\cA_\tau"]\& \pi_1{}_*\pi_3^*M.
\end{tikzcd}
\]

So we have that $\cA = J_{\tau^*M}\circ \cA_\tau-(\pi_1{}_*\pi_3^*\tau^*\cA_\tau)^{-1}\circ \nabla = J_{\tau^*M}\circ \cA_\tau-\pi_1{}_*\pi_3^*\cA_\tau\circ \nabla$ (we implicitly identify $\cA$ with $\pi_1{}_*\cA|_M$). Notice now that $J_{\tau^*M}\circ \cA_\tau = \pi_1{}_*\pi_3^* \cA_\tau \circ J$: both are equal as maps $M\to \pi_3{}_*\pi_3^*\tau^*M$ due to the adjunction relation, and $\pi_1{}_*\pi_3^*\tau^*M=\pi_3{}_*\pi_3^*\tau^*M$ as sheaves of groups. Therefore,
\[
\pi_1{}_*\cA|_M = \pi_1{}_*\pi_3^* \cA_\tau\circ (J - \nabla).
\]
Let us call $D:\pi_1^*M\to \pi_3^*M$ the map obtained from $J-\nabla$ from the adjunction $\pi_1^*\vdash \pi_1{}_*$. We obtain the relation $
\cA = \pi_3^* \cA_\tau\circ D
$. Now, $D=\pi_3^* \tau^*\cA_\tau\circ \cA$ is $\O$-linear, and further $D|_\Delta=\tau^*\cA_\tau\circ \cA_\tau=\Id$. Therefore, $\nabla$ is a covariant derivative, i.e. a linear connection on $M$, again by the reasoning in \cite{deligne}: any such $\O$-linear map $D$ which restricts to the identity on $\Delta$ gives a linear connection $\nabla = J-D$.

It remains to check that $\tau^* \nabla \circ \cA_\tau=\pi_1{}_*\pi_3^*\cA_\tau\circ \nabla$. We can repeat the same reasoning from equation~(\ref{eq:connection}) to conclude that $(\sigma^*(\tau,\tau)^*)\circ D$ acts as $-1$ on $I_\Delta \pi_1^*M$. Note that $\sigma\circ(\tau,\tau) = (\tau,\tau)\circ \sigma$ is the map that interchanges $\pi_1$ with $\pi_3$, so in this case $\sigma\circ(\tau,\tau)$ plays the role of $\sigma$ above and $\pi_3$ plays the role of $\pi_2$. Let us abbreviate $(\tau,\tau)\circ \sigma=\wt\sigma$. Taking this into account, let us show that $\wt \sigma^*D=D^{-1}$: $\pi_1^*M$ is generated by elements of the form $\pi_1^*m\in \pi_1^*M$. Then $J(m)=\pi_3^*m$, by definition, and $\nabla m\in I_\Delta \pi_3^*M$, since $J\equiv \Id\mod I_\Delta$. Therefore, 
\[
(\wt\sigma^*D)\left(D(\pi_1^*m)\right)= (\wt\sigma^*\circ D\circ \wt\sigma^*)\left(
Jm-\nabla m\right) = (\wt\sigma^*\circ D\circ \wt\sigma^*)\left(
\pi_3^*(m)-\nabla m\right) = \]\[=\wt\sigma^*(D(\pi_1^*m))-\wt\sigma^*(D\wt\sigma^*\nabla m)=\wt\sigma^*(
\pi_3^*m-\nabla m)+\wt\sigma^*\nabla m = \pi_1^*m.
\]
Then the relation $\sigma^*\cA\circ \cA = \Id$ implies the following:
\begin{align*}
\sigma^*\cA = \cA^{-1}&\Rightarrow (\pi_3^* \cA_\tau\circ D)^{-1}  =  \cA^{-1} = \sigma^*\cA= \sigma^*(\pi_3^*\cA_\tau \circ D) = \\ 
&\phantom{\Rightarrow}= \pi_1^*\tau^*\cA_\tau \circ \sigma^*D 
\overset{(\tau,\tau)^*\sigma^*D=D^{-1}}
=\pi_1^*\cA_\tau^{-1} \circ (\tau,\tau)^*D^{-1} \\
 &\Rightarrow \pi_3^* \cA_\tau\circ D = (\tau,\tau)^*D\circ \pi_1^* \cA_\tau.
\end{align*}
The last equality, after applying $\pi_1{}_*$ and restricting to $M\subset \pi_1{}_*\pi_1^* M$, reads
\[
\pi_1{}_*\pi_3^* \cA_\tau\circ (J-\nabla) = \tau^*(J-\nabla)\circ \cA_\tau.
\]
Now we note that $
\pi_1{}_*\pi_3^*\cA_\tau\circ J = \tau^* J \circ \cA_\tau
$: we observe that $\pi_3^*\cA_\tau\circ J = J_{\tau^*M}\circ \cA_\tau$ as maps to $\pi_3{}_*\pi_3^*\tau^*M$, so they are equal after identifying the latter with $\pi_1{}_*\pi_3^*\tau^*M$. Together with the fact that $\tau^*J = J_{\tau^*M}$, the above equality follows. Therefore, $\pi_1{}_*\pi_3^*\cA_\tau \circ \nabla = \tau^*\nabla \circ \cA_\tau$.

The identification between $\Omega\otimes M$ and $\pi_1{}_*I_\Delta\pi_3^*M$ takes a generator of the form $df\otimes m$ and maps it to $\pi_1{}_*((\pi_3^*f-\pi_1^*f)\pi_3^*m)$. Therefore, for any morphism $\phi:M\to N$
\begin{equation}\label{eq:connection2}
(\pi_1{}_*\pi_3^*\phi)(df\otimes m)= \pi_1{}_*((\pi_3^*f-\pi_1^*f)(\pi_3^*\phi)(\pi_3^*m)=\pi_1{}_*((\pi_3^*f-\pi_1^*f)\pi_3^*(\phi m)=df\otimes \phi m.
\end{equation}
So applying this to $\phi=\cA_\tau$ we see that for an element $\sum \alpha_i\otimes m$ with $\alpha_i\in \Omega$, the action of $\pi_1{}_*\pi_3^*\cA_\tau$ by linearity is $\pi_1{}_*\pi_3^*\cA_\tau\left( \sum \alpha_i\otimes m\right) = \sum \alpha_i\otimes \cA_\tau m$.

Consider now a local section $m\in M$, and let $\nabla m = \sum \alpha_i\otimes m_i\in \Omega\otimes M$. We have that
\[\nabla (\tau^* (\cA_\tau m)) = \tau^* (\tau^*\nabla (\cA_\tau m)) = \tau^*\left( 
\pi_3^*\cA_\tau (\nabla m)
\right)= \tau^*\left( 
\sum \alpha_i \otimes \cA_\tau m_i
\right) =\sum \tau^*\alpha_i \otimes \tau^*(\cA_\tau m_i).\]
So $\nabla \circ \tau^*\circ \cA_\tau = \tau^*\circ (\Id_\Omega \otimes \cA_\tau)\circ \nabla$, or in other words, $\tau^*\nabla \circ \cA_\tau = (\Id_\Omega\otimes \cA_\tau)\circ \nabla$. This identity is our claim.

Let us go backwards: to construct an $E$-connection starting from $\nabla$ and $\cA_\tau$, one takes $D=J-\nabla$ as above, and $\cA = \pi_3^*\cA_\tau\circ D$; and the previous reasoning shows that it is indeed an $E$-connection if $\nabla$ and $\cA_\tau$ commute in the appropriate sense.

It remains to check that these two constructions are functors. In other words, given two sheaves $M, N$ each with $\cA,\cA_\tau,\nabla$ as above, we would like to show that a morphism $\phi:M\to N$ commutes with $\cA$ if and only if it commutes with $\cA_\tau$ and $\nabla$.

Suppose $\phi$ commutes with $\cA_\tau$ and $\nabla$, i.e. $\cA\tau\circ \phi=\tau^*\phi\circ \cA_\tau$ and $\nabla\circ \phi = (\Id_\Omega\otimes \phi)\circ \nabla$. The latter equation amounts to saying that $D\circ\pi_1^*\phi=\pi_3^*\phi\circ D$: indeed, $J:\pi_1^*m\mapsto \pi_3^*m$ commutes in this way with $\phi$, and further if $\nabla m=\sum \alpha_i\otimes m_i$, we can apply (\ref{eq:connection2}) again to conclude that
\[
\pi_3^*\phi(\nabla m)=\pi_3^*\phi\left(
\sum \alpha_i\otimes m_i
\right)= 
\sum \alpha_i\otimes \phi m_i= (\Id\otimes \phi)\nabla m.
\]
We have that $\pi_1{}_*(D\circ \pi_1^*\phi)|_M=(J-\nabla)\circ \phi=\pi_1{}_*\pi_3^*\phi\circ (J-\nabla)=\pi_1{}_*(\pi_3^*\phi\circ D)|_M$, so by the adjunction we have that $D\circ\pi_1^*\phi=\pi_3^*\phi\circ D$. Finally, we have the desired relation:
\[
\cA \circ \pi_1^*\phi = \pi_3^*\cA_\tau\circ D\circ \pi_1^*\phi = \pi_3^*\cA_\tau\circ \pi_3^*\phi \circ D=  \pi_3^*\tau^*\phi \circ( \pi_3^*\cA_\tau\circ D)= \pi_2^*\phi \circ \cA.
\]
Conversely, suppose $\phi$ is such that $\cA\circ \pi_1^*\phi = \pi_2^*\phi\circ \cA$. Taking this relation restricted to $\Delta$ we obtain the equation $\cA_\tau \circ \phi = \tau^*\phi = \cA_\tau$. Now we can proceed as above:
\[
\pi_3^*\cA_\tau\circ D\circ \pi_1^*\phi =\cA \circ \pi_1^*\phi= \pi_2^*\phi \circ \cA =   \pi_3^*\tau^*\phi \circ( \pi_3^*\cA_\tau\circ D)=\pi_3^*\cA_\tau\circ \pi_3^*\phi \circ D.
\]
We conclude that $D\circ \pi_1^*\phi = \pi_3^*\phi \circ D$, from which it follows that $\nabla \circ \phi=\pi_3^*\phi\circ \nabla$, by following the reasoning above.

\end{proof}

\section{Examples}\label{sec:examples}

We will compute the local type of some rank 1 elliptic modules and we will show how one can use the local type to classify modules. In particular, Lemma~\ref{lem:allTorsionFreeModules} shows how to describe all elliptic submodules of a given elliptic module that are generically equal, in particular showing that there is a smallest such module, which we may call the intermediate extension by analogy with the $D$-module case. Further, we can also describe the elliptic modules with torsion by describing the extension groups of an elliptic module by a torsion elliptic module.

Let $k$ be algebraically closed, and let $E$ be an elliptic curve with origin $O$. Let us use $+$ to denote $E$'s group law. Fix a nontorsion point $P_0\in E(k)$. Let $\sigma_1:E\to E$ be the map $x\mapsto -x$, and consider the projection $\pi_1:E\to E/\sigma_1 \cong \P^1$. Let $\sigma:E\to E$ be the involution $x\mapsto (P_0-x)$. Let $\pi_2 = \pi_1\circ \sigma$.

\begin{lemma}
As defined above, $(\pi_1,\pi_2):E\to \P^1\times \P^1$ is an embedding, which is necessarily symmetric: by construction, $(\pi_1,\pi_2)\circ \sigma = (\pi_2,\pi_1)$.
\end{lemma}
\begin{proof}
Let us show $(\pi_1,\pi_2)$ is injective. The map $\pi_1$ identifies pairs of points $x$ and $-x$, and it is ramified at the 2-torsion points. Meanwhile, if $x\neq y$
\[
\pi_1(\sigma x) = \pi_2(x) = \pi_2(y)=  \pi_1(\sigma y) \Leftrightarrow P_0-x = \sigma x = - \sigma y = -P_0 + y \Leftrightarrow x =  2P_0 - y.
\]
Since $E$ is a group and $2P_0\neq O$, there are no $x, y\in E$ such that $x=-y$ and $x=2P_0-y$. Therefore, $(\pi_1,\pi_2)$ is injective. To see that it is an embedding, we only need to see that it is unramified. The geometric points where $(\pi_1,\pi_2)$ is ramified will be the points where both $\pi_1$ and $\pi_2$ are ramified. The former are the 2-torsion points and the latter are image of the 2-torsion points by $\sigma$. Therefore, if $(\pi_1,\pi_2)$ is ramified at $x$ we have that $2x=O$ and $2P_0-2x = O $, which implies that $P_0$ is 2-torsion, a contradiction.
\end{proof}

To study $E$-modules of generic rank $1$, we will start by studying $E$-module structures on the (pushforward of) the stalk of the generic point of $\P^1$, i.e. the sheaf whose sections on any nonempty open set equal $k(z)$. Then, we will look at torsion-free $E$-modules of generic rank $1$, i.e. modules which embed into $k(z)$. Finally, we will consider $E$-modules with torsion.

\begin{notation}
Let $z$ be a coordinate on $\P^1$, and let $(z_1,z_2)$ be coordinates on $\P^1\times \P^1$ pulled back from $z$. Let $k(z)$ (resp. $k(E)$) be the field of rational functions of $\P^1$ (resp. $E$). For a point $x\in E(k)$, we will use $(x)$ to denote the corresponding divisor. Let $G$ be the infinite dihedral group generated by $\sigma$ and $\sigma_1$. Let $S$ be a set of representatives of $E(k)/G$, chosen such that it contains every $x$ such that $2x=P_0$ and it contains the 2-torsion points.

We will use Proposition~\ref{prop:EllipticModulesAreDihedral} to see $E$-modules as $G$-equivariant modules on $E$. Recall that to define the local type we have to take a generator $\tau$ of $\Z\subseteq G$. We will let $\tau = \sigma\sigma_1:x\mapsto x+P_0$.
\end{notation}

\begin{remark}
The conjugacy classes of $G$ have a set of representatives given by $\{\sigma_1,\sigma\}\cup \{\tau^n\mid n\in \Z\}$. Since the action of $\{\tau^n\}$ is free on $E(k)$, the only nontrivial stabilizers of points of $E$ are conjugate to $\{1,\sigma_1\}$ or $\{1,\sigma\}$, i.e. the only orbits which are not in bijection with $G$ are those of the 2-torsion points and the points $x$ such that $2x=P_0$.
\end{remark}

\begin{lemma}\label{lem:niceDivisor}
Let $M$ be the pushforward to $\P^1$ of a rank 1 free module on the generic point of $\P^1$, i.e. the module $k(z)\cdot s$ on every open set.

Every $E$-module structure on $M$ can be represented in the following way. There exists a generator $s\in M$ such that $\cA\pi_1^* s = f \pi_2^*s$ for a function $f\in k(E)$ such that $f\cdot f^\sigma = 1$ and further there exist some integers $n_x$ such that
\begin{equation}\label{eq:niceDivisor}
\div f = \sum_{x\in S} n_x \cdot (x) - \sum_{x\in S} n_x\cdot(P_0-x).
\end{equation}
Further, if $x\neq O$ and $2x=O$, we can choose $n_x\in \{0,1\}$;	 and if $(2x)=P_0$, we have $n_x=0$. With these restrictions, all the coefficients $n_x$ are uniquely determined.

In all the cases, the automorphism group of $M$ is $k^\times$.
\end{lemma}

\begin{proof}
An $E$-module structure on $M$ consists of an isomorphism $\cA:\pi_1^* M \to \pi_2^*M$ such that $\sigma^*\cA\circ \cA = \Id$. Let $k(E)$ be the field of rational functions on $E$. As quasicoherent sheaves, $\pi_1^*M\cong \pi_2^*M\cong k(E)$. Therefore, $\cA$ is given by a choice of some $f\in k(E)$, and letting $\cA(\pi_1^*s) = f\pi_2^*s$, subject to the condition that $f^\sigma f= 1$.

Recall that $k(E)^\times/k^\times$ is isomorphic to the group of principal divisors on $E$, $\Prin(E)$. The equation $ff^\sigma =\Id$ implies that the divisor of $f$ has a similar relation. Namely, for some $n_j\in \Z$ and some $x_j\in E$,
\[
\div f = \sum_j n_j x_j - n_j(P_0-x_j).
\]
Therefore, rank $1$ elliptic modules up to constant are parametrized (non-uniquely) by $\{D\in \Prin(E)\mid \sigma(D) = -D\}$.

Let us see the effect of gauge transformations. An automorphism of $M$ is given by sending $s\mapsto g(z)s$ for some $g\in k(z)$. Letting $\wt s= g(z)s$, and denoting by $(z_1,z_2)$ the coordinates of $\P^1\times \P^1$, we have:
\[
\cA (\pi_1^* \wt s) = \cA( g(z_1)\pi_1^*s) = g(z_1)f\pi_2^*s = \frac{g(z_1)}{g(z_2)} f\pi_2^*\wt s.
\]
So gauge transformations take the form $f\mapsto \wt f = \frac{g(z_1)}{g(z_2)} f = \frac{g}{g^\sigma} f$ for some $g\in k(z)$. In terms of divisors, $\div g$ is a divisor in $\Prin(E)$ which is $\sigma_1$-invariant, and we will have that $\div \wt f = \div g-\sigma(\div g) + \div f$. Therefore, if we let $\sim$ we the equivalence relation on elliptic modules generated by isomorphisms and multiplication of $\cA$ by constants (necessarily $\pm 1$), we have:
\begin{equation}\label{eq:divisorQuotient}
\frac{\{
\text{Elliptic module structures on M}
\}}{\sim} \cong
\frac{\{D\in \Prin(E)\mid \sigma(D) = -D \}}{\{D-\sigma(D)\mid D=\sigma_1(D), D\in \Prin(E) \}}.
\end{equation}
Let us denote $H = \{D-\sigma(D)\mid D=\sigma_1(D), D\in \Prin(E) \}$. The statement we are trying to prove is that the quotient on the right hand side is generated by elements of the form (\ref{eq:niceDivisor}). We will show that this is the case for the larger group $
\frac{\{D\in \Div(E)\mid \sigma(D) = -D \}}{H}$. Equivalently, we can show the same statement replacing $H$ by $\wt H = H + \langle(O)-(P_0)\rangle$, since $(O)-(P_0)$ is included in the generators we are looking for.

Principal divisors invariant by $\sigma_1$ are divisors pulled back from $\P^1$ via $\pi_1$, so they are generated by divisors of the form $(x) + (-x) - 2 (O)$. Therefore, $H$ is generated by divisors of the form
\[
D_x = (x) + (-x) - 2 (O) - \sigma \left( (x) + (-x) - 2 (O)\right) = (x) + (-x) - (P_0-x) - (P_0+x) - 2(O) + 2(P_0).
\]


Let us consider an orbit $G\cdot x$, and consider a divisor $D$ supported on $G\cdot x$ such that $D = -\sigma(D)$. Let us show that there exists $n\in \Z$ such that
\[
D \equiv n(x)-n(P_0-x)  \mod \wt H.
\]
We can do this by induction. If $D = -\sigma(D)$, then $D$ is a linear combination of divisors of the form $(mP_0+x) - ((1-m)P_0-x)$ for $m\in \Z$. Note that
\[
D_{(m-1)P_0+x} \equiv ((m-1)P_0+x) + ((1-m)P_0-x) - ((2-m)P_0-x) - (mP_0+x) \mod ((O)-(P_0)) \Rightarrow\]\begin{equation}\label{eq:divisorRel}
(mP_0+x) - ((1-m)P_0-x) \equiv ((m-1)P_0+x)  - ((2-m)P_0-x) \mod \wt H.
\end{equation}

Iterating this relation, we have that $(mP_0+x) - ((1-m)P_0-x) \equiv (x) - (P_0-x)\mod \wt H$, as desired. In the particular case where $2x=P_0$, then it will follow that $(mP_0+x) - ((1-m)P_0-x) \equiv 0$. If $x\neq O$ is 2-torsion, then $D_x\equiv 2(x) - 2(P_0-x) \mod \wt H$, so further we can reduce $n_x$ modulo 2, as desired.

Now let us discuss the uniqueness. We will construct homomorphisms with domain $\Div(E)/H$ and show that all of them together determine the coefficients $n_x$. For any $x\in S$, consider the following map $\Div(E)\to \Z$:
\[
\sum_{n\in \Z} a_n(x+nP_0)+ \sum_{Q\notin x+\Z P_0} b_Q (Q) \longmapsto \sum_{n\in \Z} a_n.
\]
We can see that the map above is well-defined on the equivalence classes modulo $H$, i.e. that it sends $D_y$ to $0$ for every $y\in E$. Therefore, the number $n_x\in \Z$ in (\ref{eq:divisorRel}) uniquely determined.

It remains to consider $n_x$ for $(2x)=(O)$. If $x\neq O$, we use the following invariant with values in $\Z/4\Z$:
\[
\sum_{n\in \Z} a_n(x+nP_0)+ \sum_{Q\notin x+\Z P_0} b_Q (Q) \longmapsto \sum_{n\in \Z} (-1)^na_n \in \Z/4\Z.
\]
Lastly, if $x=O$, we can verify that the following map gives an invariant. Write every point as $\epsilon x + nP_0$, where $\epsilon\in \pm 1$ and $x\in S$. If $2x = nP_0$, then assume $\epsilon = 1$, and let:
\[
\epsilon x + nP_0\mapsto n
\]
We can verify directly that the above map sends $D_{(\epsilon x+nP_0)}$ to $0$ for every $\epsilon$ and $P_0$, and therefore it is an invariant. This invariant applied to (\ref{eq:niceDivisor}) yields
\[
\sum_{x\in S} n_x \cdot (x) -  n_x\cdot(P_0-x) \mapsto -n_O - \sum_{\substack{ x\in S\\ x\neq O} } n_x .
\]
In particular, we have shown that $n_x$ is uniquely determined for all $x$ except for $x=O$, and this invariant shows that $n_O$ is uniquely determined as well.

It only remains to verify that the automorphism group is $k^\times$. Any $\C(z)$-linear automorphism of $M$is given by multiplication by $g\in k(z)$. For multiplication to be an $E$-module morphism, we require that $\pi_2^*g \cA = \cA \pi_1^*g$. Let $D=\div g$. The previous equation implies that $\pi_1^*D = \pi_2^*D$, in other words, $ \sigma^*\pi_1^*D = D$. Since $\sigma_1^*\pi_1^*D = \pi_1^*D$, this implies that $\pi_1^*D$ is $G$-invariant. Since $G$ has no finite orbits, this implies that $D=0$, so $g$ must be a constant function.

\end{proof}

Now that we have listed all the elliptic module structures on $k(z)$, let us compute their local types. For any point $p\in E$, we will let $R_p$ be its completed local ring and $K_p$ will be the field of fractions of $R_p$.

\begin{lemma}
As above, let $M = k(z)\cdot s$ be an elliptic module, with $\cA \pi_1^*s = f\pi_2^* s$ and such that
\[
\div f = \sum_{x\in S} n_x \cdot (x) - \sum_{x\in S} n_x\cdot(P_0-x).
\]
Choose $p\in S$, let $\rho$ be a local generator of the maximal ideal at $p$, and let $R_p \cong k[[\rho]]$ be the completed local ring at $p$. Consider $\pi_1^*M$ as a $G$-equivariant sheaf. Then the local type at $p$ is determined as follows:
\begin{enumerate}
\item If $\St_p = \{1\}$, $M|_{\pz}^l$ is generated over $R_p$ by $\{ \rho^{-n_p}\pi_1^*s\}$ and $M|_{\pz}^r$ is generated by $\{\pi_1^*s \}$.
\item If $\St_p = \{1,\sigma\}$, $M|_{\pz}^{lr}$ is generated over $R_p$ by $\{\pi_1^*s\}$.
\item If $\St_p = \{1,\sigma_1\}$, $M|_{\pz}^{lr}$ is generated over $R_p$ by $\{ \rho^{-n_p}\pi_1^*s\}$.
\end{enumerate}
\end{lemma}

\begin{proof}
Let $L\subseteq \pi_1^*M$ be the rank 1 free $\O_E$-module generated by $s_1\coloneqq \pi_1^*s$. Our goal is to find the stalk at $p$ of $\ov\tau^n$ for $n\gg 0$ and $n\ll 0$. First, note that
\[
\ov \tau s_1 =\ov\sigma \ov\sigma_1 \pi_1^*s = \sigma^*\circ \cA_\sigma \circ \sigma_1^*\circ \cA_{\sigma_1} s_1 = \sigma^*\circ \cA_\sigma \pi_1^*s = \sigma^* (f \pi_2^*s)  = f^\sigma \sigma^*\pi_2^*s = f^\sigma s_1.
\]
Now, by induction, we have that if $n > 0$, $\ov\tau^n s_1 = f^{\sigma\tau^{-n+1}}f^{\sigma\tau^{-n+2}}\cdots f^\sigma s_1$:
\[
\ov\tau (\ov\tau^n s_1) = \ov \tau (f^{\sigma\tau^{-n+1}}f^{\sigma\tau^{-n+2}}\cdots f^\sigma s_1) = (f^{\sigma\tau^{-n+1}})^{\tau^{-1}} (f^{\sigma\tau^{-n+2}})^{\tau^{-1}}\cdots (f^\sigma)^{\tau^{-1}}\ov\tau s_1 = f^{\sigma\tau^{-n}}f^{\sigma\tau^{-n+1}}\cdots f^\sigma s_1.
\]
From the equation $\ov\tau s_1 = f^\sigma s_1$, we can also conclude that $\ov\tau^{-1} s_1 = (f^{-1})^{\sigma_1} s_1$, and analogously, if $n>0$, $\ov \tau^{-n}s_1= (f^{-1})^{\sigma_1\tau^{n-1}} \cdots (f^{-1})^{\sigma_1} s_1$. Now, note that for any $g\in G$, $\div f^g = \div (f\circ g) = g^{-1}(\div f)$.

Then the support of $\div f^g$ contains $p$ if and only if $g(p) = p$ or $p = g^{-1}(P_0-p) = g^{-1}\sigma (p)$, i.e. if $g$ or $\sigma g$ stabilize $p$.

Suppose that the stabilizer of $p$ is trivial. Then, every function of the form $f^{\sigma \tau^{m}}$ and $f^{\sigma_1 \tau^{m-1}}$ has no zeroes or poles at $p$ unless $m=0$, so it is a unit of $R_p$. By the formulas above, this shows that for $n\ge 1$, $\ov \tau^n s_1$ differs from $\ov\tau s_1$ by multiplication by a unit in $R_p$, and for $n\le 0$, $\ov \tau^n s_1$ equals $s_1$ up to multiplication by a unit. Therefore, if $n\gg 0$,
\[
M|_{\pz}^l = (\tau^n M)_p = R_p \cdot f^\sigma s_1 = R_p \cdot \rho^{-n_p} s_1;\quad
M|_{\pz}^r = (\tau^{-n} M)_p = R_p \cdot  s_1.
\]
If $\sigma p =p$, then the $f$ we have chosen has no zeroes or poles on $G\cdot p$, so $M|_{\pz}^{lr} = R_p\cdot s$. Finally, if $\sigma_1 p = p$, then the following four functions may have zeroes or poles at $p$: $f,f^\sigma,f^{\sigma_1}, f^{\sigma \sigma_1}$. Repeating the reasoning above, for $n\gg 0$
\[
M|_{\pz}^{lr} = (\tau^n M)_p = R_p \cdot f^\sigma s_1 =  R_p \cdot\rho^{-n_p} s_1 .
\]
\end{proof}

Using the local type, we can easily describe the $E$-submodules of $M$. We can in fact describe all the $G$-equivariant subsheaves of $\pi_1^*M$, and using Proposition~\ref{prop:EllipticModulesAreDihedral}, the $E$-modules will be found among these.

\begin{lemma}\label{lem:allTorsionFreeModules}
Let $M\in \Hol(E)$, and let $S$ be a set of representatives of $E(k)/S$ as above. The following sets are in bijection:
\begin{enumerate}
\item The set of $G$-equivariant sheaves $M'$ of $M$ such that $M/M'$ is a torsion sheaf.
\item The collections $\{M'_p\mid p\in S \}$, where $M'_p\subseteq M_p$ is an $R_p$-submodule preserved by the induced action of $\St_p$ and such that $M'_p\supseteq M|_{\pz}^\star$ for all meaningful $\star = l,r,lr$.
\end{enumerate}
The bijection is given by taking a module $M'$ and considering the collection of its formal stalks at $p\in S$.
\end{lemma}

\begin{proof}
To see that the given map from modules to collections of stalks is well-defined, we only need to verify the claim that $M'_p\supseteq M|_{\pz}^\star$. This follows from Proposition~\ref{prop:wellDefined}. If $M/M'$ is a torsion sheaf, we see directly that $(M/M')|_{\pz}^\star = 0$ for $\star = l,r,lr$. By part (4) of Proposition~\ref{prop:wellDefined}, this means that the inclusion $M'\to M$ induces isomorphisms $M'|_{\pz}^\star \cong M|_{\pz}^\star$. In particular,
$
M'_p \supseteq M'|_{\pz}^\star =M|_{\pz}^\star$.

Now, let us show that the map from modules to stalks is injective. Suppose we have $M',M''\subseteq M$ as in the statement, whose stalks agree on $S$. Then, using the action of $G$ we can see that their stalks agree on every point of $E$, so $M'=M''$, as a sheaf is determined by its stalks.

Finally, it remains to show that the map is surjective. Let us start by showing that there exists an $M'$ with the smallest possible stalks, i.e. $M'_p = M|_\pz^l+M|_\pz^r$ for $\St_p=1$ and $M'_p = M|_\pz^{lr}$ for $\St_p\neq 1$. Let us denote use $j_{!*}M_p$ to denote either $M|_\pz^l+M|_\pz^r$ or $M|_\pz^{lr}$, depending on $\St_p$. We will use Proposition~\ref{prop:EllipticModulesAreDihedral} to show instead the existence of $\pi_1^*M'$ and work with $G$-equivariant sheaves on $E$.

Let $L$ be any coherent subsheaf of $M$ (not necessarily preserved by $G$) that equals $M$ over the generic point. We claim that $L$ satisfies $L_p\subseteq j_{!*}M_p$ on every point away from a finite set. Since $L$ is coherent, $\ov\tau L$ is also coherent, so the modules $(L + \ov\tau L)/L$ and $(L + \ov\tau L)/\ov\tau L$ are both torsion and coherent, hence finite. This implies that the stalks of $\ov\tau L$ and $L$ differ only at a finite set. Iterating this reasoning, we have that for $p$ away from a (fixed) finite set of $G$-orbits, $(\ov\tau^nL)_p = L_p$ for all $n$. Therefore, $L_p = M|_{\pz}^\star$ on these orbits, by definition of $M|_{\pz}^\star$. Let us refer to the remaining orbits as the ``bad'' orbits.

Let $N$ be the smallest $G$-equivariant sheaf containing $L$, i.e. $N = \sum_{g\in G} \ov g L$. By the construction, the stalks of $N$ outside of a finite number of orbits agree with $j_{!*}M_p$. Now, let $p$ be a point in one of the bad orbits, let $E^*\coloneqq E\setminus Gp$, and consider the following object in $\Loc \times_{\Locc} \Hol(E^* )$: in $\Loc$, take $j_{!*}M_p\subseteq M|_{\pz}$, which is an object of $\Loc$ by Proposition \ref{prop:wellDefined}. In $\Hol(E^* )$, take $N|_{E^* }$, and glue them via the isomorphism $N|_{E^* }|_\pz \cong M|_{\pz} \cong K_p\otimes j_{!*}M_p$ that comes from $M$. By Theorem~\ref{thm:main}, there is some $N'\in \Hol(E)$ mapping to this pair of objects. It is a subobject of $M$, via the pair of maps $N'|_{\pz}=j_{!*}M_p\subseteq M|_{\pz}$ and $N'|_{E^*} \cong N|_{E^*} \subseteq M|_{E^*}$. Further, it agrees with $N$ outside of $Gp$, and on $Gp$ it has stalk equal to $j_{!*}M_p$. Thus, going from $N$ to $N'$ we have reduced the number of bad orbits by $1$, so we can repeat this until there are no bad orbits remaining. Let us call the resulting submodule $j_{!*}M$.

%
%

Now, let us show that the map in the statement is surjective: by the discussion above, any $E$-submodule of $M$ must contain $j_{!*}M$. Therefore, modules contained in $M$ and containing $j_{!*}M$ are in bijection with submodules of the torsion module $T\coloneqq M/j_{!*}M$. Being a torsion sheaf, $T$ is isomorphic as a quasicoherent sheaf to the direct sum of its stalks, and for every $g\in G$, $\cA_{g}$ induces an isomorphism $T_p\xrightarrow[\sim]{(\cA_g)_p} (g^*T)_p\cong g^*(T_{gp})$. It follows that as an $E$-module, $T$ splits as a direct sum:
\[
T = \bigoplus_{Gp\in E/G} T_{Gp};\quad \supp T_{Gp}\subseteq Gp.
\]
By restricting Theorem~\ref{thm:main} to the full subcategory of torsion modules supported on $Gp$, we can see that giving a torsion module is equivalent to giving an object $(M_p,M_{E^*},\cong)\in \Loc \times_{\Locc} \Hol(E^* )$, with the condition that $M_p$ is torsion and $M_{E^*}=0$, i.e. $E$-modules supported on $Gp$ are equivalent to torsion $R_p$-modules with an equivariant structure. In particular, submodules of $T_{Gp}$ are in bijection with submodules of its stalk at $p$. Putting all the orbits together and using the fact that stalks commute with direct sums, we have concluded the proof.
\end{proof}

The previous proposition allows us to list all the submodules of $k(z)$ with any elliptic module structure, and obtain their local types. It remains to describe general modules which are generically $k(z)$. These are extensions of torsion-free modules by torsion modules. We begin by describing these extensions on formal neighborhoods.

\begin{lemma}\label{lem:extensionsSt=1}
Let $p\in E$ with $\St_p = \{1\}$. Suppose $M\in \Loc$ is torsion-free and finitely generated, and $T$ is a finitely generated torsion $R_p$-module, which can be seen as an element of $\Loc$ by (necessarily) letting $T^\star = 0$. Then extensions of $M$ by $T$ are classified by
\[
\Ext^1_{\Loc}(M,T) \cong \frac{\Hom_{R_p}(M^l,T)\oplus \Hom_{R_p}(M^r,T)}{\Hom_{R_p}(M,T)},
\]
Where $\Hom_{R_p}(M,T)$ maps into the two groups in the numerator by restriction from $M$ to $M^\star$. Therefore, the set of isomorphism classes of objects of $\Loc$ which are extensions of $M$ by $T$ is in bijection with the quotient
\[
\frac{\Ext^1_{\Loc}(M,T)}{\Aut_{\Loc}(M)\times \Aut_{R_p}(T)},
\]
where the automorphism groups act by (pre)composition.
\end{lemma}
\begin{proof}
The statement about $\Ext^1_{\Loc}(M,T)$ is a particular case of \cite[Proposition 3.15]{H}. Consider two extensions in $\Loc$, $0\to T\xrightarrow{i_1} N_1\xrightarrow{p_1} M\to 0$ and $0\to T\xrightarrow{i_2} N_2\xrightarrow{p_2} M\to 0$ and an isomorphism $\phi \colon N_1\cong N_2$. Since $T$ is the torsion submodule of $N_1$ and $N_2$, $\phi$ must preserve it, i.e. there must be an isomorphism $\ov\phi\in \Aut_{\Loc}(T)$ such that $\phi\circ i_1= i_2\circ \ov\phi$, and therefore $\phi$ induces an automorphism $\wt\phi$ on the quotients $M\cong N_i/T$, i.e. $\wt\phi\circ p_1 = p_2\circ \phi$. This shows that the extension classes of both extensions are in the same orbit of $\Aut_{\Loc}(M)\times \Aut_{R_p}(T)$, as desired.
\end{proof}

\begin{lemma}\label{lem:extensionsStNot1}
Let $p\in E$ with $\St_p = \{1,\sigma\}$. Suppose $M\in \Loc$ is torsion-free and finitely generated as an $R_p$-module, and $T$ is a finitely generated torsion $R_p$-module with a $\Z/2\Z\langle \sigma\rangle$, which can be seen as an element of $\Loc$ by (necessarily) letting $T^{lr} = 0$. Then extensions of $M$ by $T$ are classified by
\[
\Ext^1_{\Loc}(M,T) \cong \frac{ \Hom_{R_p}^{\Z/2\Z}(M^{lr},T)}{\Hom_{R_p}^{\Z/2\Z}(M,T)},
\]
where the map is given by restriction from $M$ to $M^{lr}$. Therefore, the set of isomorphism classes of objects of $\Loc$ which are extensions of $M$ by $T$ is in bijection with the quotient
\[
\frac{\Ext^1_{\Loc}(M,T)}{\Aut_{\Loc}(M)\times \Aut_{R_p}^{\Z/2\Z}(T)},
\]
where the automorphism groups act by (pre)composition.
\end{lemma}
\begin{proof}
Let us start by showing that $\Ext^1(M,T) = 0$ as equivariant $R_p$-modules. Let $t\in R_p$ generate the maximal ideal, so we have that $\sigma^*t = -t$. Consider the action of $\ov\sigma$ on $M/tM$. Since $\ov\sigma^2= \Id$, $M/tM$ splits into two eigenspaces with eigenvalues $1$ and $-1$. Choose a $k$-basis of $M/tM$ composed of eigenvectors. Let us show that every element $m$ of this basis lifts to an eigenvector of $\ov\sigma$ in $M$, by showing it lifts to a basis of $M/t^nM$ by induction on $n$. Suppose we have a lift $m_n\in M/t^nM$ such that $\ov\sigma m = \epsilon m$, where $\epsilon = \pm 1$. Replacing $\ov\sigma$ by $-\ov\sigma$, we can prove this statement assuming that $\epsilon = 1$. Take any lift $\wt m$ of $m_n$ to $M/t^{n+1}M$. Note that $a\coloneqq \ov\sigma\wt m-  \wt m\in t^nM/t^{n-1}M$ satisfies $\ov\sigma a = -a$. It follows that the desired lift is $\wt m + \frac{a}{2}$, as we can check directly.

Let $R_p\cdot s^+$ (resp. $R_p\cdot s^-$) be the rank 1 free $R_p$ module generated by $s^+$ (resp. $s^-$) with equivariant structure given by $\ov\sigma s^+ = s^+$ (resp. $\ov\sigma s^-=-s^-$). By the previous paragraph, as an equivariant $R_p$-module, $M$ is a direct sum of copies of $R_p^+$ and $R_p^-$. Therefore, to show that $\Ext^1(M,T) = 0$ it suffices to verify this for $M=R_p^+,R_p^-$. Consider an extension $0\to T \to \ov M\to R_p^+\to 0$. Since $R_p^+$ is a free $R_p$-module, there is a $R_p$-linear section $s^+\to \wt{ s}$. It is $\Z/2\Z$-equivariant if $\ov\sigma \wt{s} =  \wt{s}$. Proceeding as before, the following map is necessarily a $\Z/2\Z$-equivariant section:
\[
s^+\mapsto \frac{\wt{s} + \ov\sigma \wt{s}}{2}.
\]
Note that since the map $\ov M\to R_p^+$ is $\Z/2\Z$-equivariant, we have that the image of $\frac{\wt{s} + \ov\sigma \wt{s}}{2}$ is $s^+$, as desired. To show that $\Ext^1(M,R_p^-)=0$, we can proceed analogously to conclude that $s^-\mapsto \frac{\wt s - \ov\sigma \wt s}{2}$ is a $\Z/2\Z$-equivariant section.
%

Therefore, all extensions in $\Ext^1_{\Loc}(M,T)$ split as extensions of $\Z/2\Z$-equivariant modules, i.e. they all take the form $0\to T\to T\oplus M\to M\to 0$, and they are determined by choosing a submodule $P^{lr}\subseteq T\oplus M$ mapping isomorphically to $M^{lr}$. Consider one such extension, and let $i \colon M^{lr}\to P^{lr}\subseteq T\oplus M$ be the inverse of the projection restricted to $P^{lr}$. The component of $i$ mapping into $M$ must be the inclusion of $M^{lr}$ in $M$, so $i = (i',1)\colon M^{lr}\to T\oplus M$. Conversely, for every $\Z/2\Z$-equivariant map $i'\colon M^{lr}\to L$ we obtain a submodule $P^{lr}$ as the image of $(i',1)$. This way we obtain a surjection $\Hom_{R_p}^{\Z/2\Z}(M^{lr},T)\twoheadrightarrow \Ext^1_{\Loc}(M,T)$.

Let us show that this surjection is $R_p$-linear: it commutes with multiplication by $R_p$, since on both groups it is induced by the action of $R_p$ on $T$. To see that it commutes with sums, we can see it directly by using the Baer sum. This is the same reasoning as the one in the proof of \cite[Proposition 3.15]{H}.

Therefore, we only need to compute the kernel of $\Hom_{R_p}^{\Z/2\Z}(M^{lr},T)\twoheadrightarrow \Ext^1_{\Loc}(M,T)$. These are the maps $i'\colon M^{lr}\to L$ for which there is a section $j:M\to T\oplus M$ which is a morphism in $\Loc$. Such a section $j$ must be of the form $(j',1)$, where $j'$ is a map $M\to T$. To be a morphism in $\Loc$, $j$ must be $\Z/2\Z$-equivariant (equivalently, $j'$ must be $\Z/2\Z$-equivariant) and we must have $j(M^{lr})\subseteq P^{lr}$. Now, note that
\[
jM^{lr}\subseteq P^{lr}\Leftrightarrow (j-i)M^{lr}\subseteq P^{lr} \xLeftrightarrow{(j'-i',0) = j-i} j'|_{M^{lr}}-i'= 0.
\]
In conclusion, the elements of $\Hom_{R_p}^{\Z/2\Z}(M^{lr},T)$ that yield a split extension are the ones which are in the image of the restriction from $\Hom_{R_p}^{\Z/2\Z}(M,T)$. The remainder of the proof is the same as the proof of Lemma~\ref{lem:extensionsSt=1}.
\end{proof}

\begin{corollary}
Let $P_0, M, f,S$ be as in Lemma~\ref{lem:niceDivisor}. Let $ M'\subseteq M$ be a subsheaf given by a collection of submodules $\{M'_p|p\in S\}$ as in Lemma~\ref{lem:allTorsionFreeModules}, where $M'_p$ is chosen to be finitely generated. Let $T$ be a torsion $E$-module. $T$ is given by specifying a torsion sheaf supported on $x\in S$ with $2x\neq O,P_0$ and a $\Z/2\Z$-equivariant torsion sheaf at the points where $2x \in \{O,P_0\}$. Let us suppose that $T$ is supported on a finite number of $\Z$-orbits.

The extensions of $M'$ by $T$ are classified by an extension class of $M'_p$ by $T_p$ for every $p\in E/G$, as described in Lemmas~\ref{lem:extensionsSt=1} and \ref{lem:extensionsStNot1}.
\end{corollary}

\begin{proof}
$T$ has the for $\bigoplus_{i=0}^p T_{i}$, where each $T_{i}$ supported on a different orbit $G\cdot p_i$. Therefore, $\Ext^1_{\Hol(E)}(E,T)\cong \bigoplus_{i} \Ext^1_{\Hol(E)}(E,T_{i})$. Applying Theorem~\ref{thm:main} and the fact that $T_{i}|_{E\setminus G\cdot p_i}=0$, we have the desired result:
\[\Ext^1_{\Hol(E)}(E,T)\cong \bigoplus_{i} \Ext^1_{{\Mod(U_{p_i})}
}(E|_{U_{p_i}},T_{i}).\]
\end{proof}

\bibliographystyle{alpha}
\bibliography{Bibliography}

\end{document}